\documentclass[11pt]{article}
\usepackage{amscd,amssymb,amsmath,verbatim,amsthm}
\usepackage{hyperref}
\usepackage{times}
\usepackage{enumerate}
\usepackage[11pt,curve,matrix,arrow,frame,graph]{xy}
\newcounter{intro}

\newtheorem{thm}{Theorem}[section]
\newtheorem{lem}[thm]{Lemma}
\newtheorem{prop}[thm]{Proposition}
\newtheorem{cor}[thm]{Corollary}
\newtheorem{defi}[thm]{Definition}

\newtheorem{rem}[thm]{Remark}

\numberwithin{equation}{section}   

\newcounter{counteroman}

\newcommand{\cref}[1]{Corollary~\ref{#1}}

\newcommand{\R}{\mathbb{R}}

\newcommand{\mcC}{\mathcal{C}}

\newcommand{\mcM}{\mathcal{M}}

\newcommand{\bpm}{\begin{pmatrix}}
\newcommand{\epm}{\end{pmatrix}}

\DeclareMathOperator{\scal}{Scal}

\DeclareMathOperator{\Lip}{Lip}

\DeclareMathOperator{\vol}{vol}

\numberwithin{equation}{section}

\renewcommand{\tilde}{\widetilde}

\renewcommand{\hat}[1]{\widehat{#1}}

\newcommand\Vol{\operatorname{Vol}}

\newcommand\paperintro%
        {%
         }
\newcommand\paperbody%
        {%
         }

\newcommand\sI{\mathcal{I}}

\DeclareMathAlphabet{\mathpzc}{OT1}{pzc}{m}{it}

\newcommand{\HH}{\mathbb H}
\newcommand{\NN}{\mathbb N}
\newcommand{\RR}{\mathbb R}

\newcommand{\del}{\partial}

\newcommand{\phg}{{\mathrm{phg}}}
\newcommand{\e}{\epsilon}
\newcommand{\calA}{{\mathcal A}}
\newcommand{\calC}{{\mathcal C}}

\newcommand{\calI}{{\mathcal I}}

\newcommand{\calL}{{\mathcal L}}
\newcommand{\calM}{{\mathcal M}}

\newcommand{\calO}{{\mathcal O}}

\newcommand{\calU}{{\mathcal U}}
\newcommand{\calV}{{\mathcal V}}

\newcommand{\tilg}{{\tilde{g}}}
\newcommand{\ice}{{\mathrm{ice}}}

\begin{document}
\title{The Yamabe problem on stratified spaces}
\author{Kazuo Akutagawa \thanks{akutagawa\@@math.is.tohoku.ac.jp}\\ Tohoku University \and 
Gilles Carron \thanks{Gilles.Carron\@@math.univ-nantes.fr} \\ Universit\'e de Nantes  \and 
Rafe Mazzeo \thanks{mazzeo\@@math.stanford.edu}\\ Stanford University}

\date{}




\maketitle

\begin{abstract}
We introduce new invariants of a Riemannian singular space, the local Yamabe and Sobolev constants,
and then go on to prove a general version of the Yamabe theorem under that the global Yamabe invariant
of the space is strictly less than one or the other of these local invariants. This rests on a small number
of structural assumptions about the space and of the behavior of the scalar curvature function on its
smooth locus. The second half of this paper shows how this result applies in the category of 
smoothly stratified pseudomanifolds, and we also prove sharp regularity for the solutions on these spaces.
This sharpens and generalizes the results of Akutagawa and Botvinnik \cite{AB} on the Yamabe problem 
on spaces with isolated conic singularities. 
\end{abstract}

\section*{Introduction}
Our aim in this paper is to study a version of the Yamabe problem on a class of compact Riemannian singular spaces 
satisfying a small list of general structural axioms which we call `almost smooth metric-measure spaces'. This 
approach emphasizes the centrality of Sobolev inequality, and indeed relies on little else. Our main existence result 
is the analogue of that part of the resolution of this problem on compact smooth manifolds $(M,g)$ obtained through the
work of Yamabe, Trudinger and Aubin, \cite{Au}, \cite{T}, \cite{Y}. In that original setting, the work of these authors 
established the existence of a smooth positive function minimizing the Yamabe functional 
\begin{equation}
Q_g(u) = \frac{ \int_M (|\nabla u|^2 + \frac{n-2}{4(n-1)} \scal_g u^2)\, dV_g}{ \left(\int_M u^{\frac{2n}{n-2}}\, dV_g\right)^{\frac{n-2}{n}} },
\label{defQ}
\end{equation}
where $\scal_g$ is the scalar curvature of the metric $g$, provided the infimum of this functional, the so-called Yamabe 
invariant (sometimes also called the Yamabe constant or conformal Yamabe invariant) of that conformal class $Y(M,[g])$, 
is strictly less than the corresponding invariant of the round sphere.  
In some papers on this subject, the energy $Q_g$ is replaced by $( (4(n-1)/(n-2))Q_g$. 
The geometric meaning of this functional is that if $u$ is a minimizer, or indeed any critical point, then the conformally related 
metric $\tilde{g} = u^{\frac{4}{n-2}}g$ has constant scalar curvature on any open set where $u > 0$. We refer to the
well-known survey paper by Lee and Parker \cite{LP}, as well as \cite{RS}, \cite{SY}, for all details on the complete existence theory 
in the setting of smooth compact manifolds.

The singular spaces $(M,g,\mu)$ we are interested in here are typically Riemannian pseudomanifolds, and in particular 
Riemannian smoothly stratified spaces with iterated edge metrics, endowed with a measure which is a smooth
positive multiple of the Riemannian volume form $dV_g$. However, as indicated already, we require
only a few structural assumptions and so our main existence theorem holds in much more general settings.
In a companion to this paper we explore this direction further, extending this method to handle rather general
semi-Riemannian spaces, for example.  The ability to allow for a more general measure $\mu$ is perhaps useful,
but plays essentially no role in any of the arguments below, and for reasons of notational simplicity, we
often omit $\mu$ altogether from the discussion. The spaces here are `mostly smooth' in that they possess an open dense
set $\Omega$ which is a smooth $n$-dimensional manifold carrying a Riemannian metric.  Infinitesimally,
every point in $\Omega$ looks the same as every other. However, that is not true if one includes the singular points.
To accomodate this, we replace the global Yamabe invariant by a new invariant which we call the {\it local} Yamabe 
invariant $Y_\ell(M,[g])$.  Briefly, this is just the infimum over all points $p\in M$ of the Yamabe invariants of 
arbitrarily small balls around $p$, where we minimize the standard energy functional amongst functions on these
balls which vanish on the outer boundaries, but not necessarily near the singular set of $M$. We also introduce 
the corresponding local Sobolev invariant $S_\ell(M,g)$. Our main existence theorem states that under various sets 
of conditions on the scalar curvature $\scal_g$ (which we regard as a function computed in the usual way on the
smooth domain $\Omega$), if the {\it global} Yamabe invariant $Y(M,[g])$ is {\it strictly} less than $Y_\ell(M,[g])$ (or, 
in some versions of the result, than $S_\ell(M,g)$), then $Q_g$ admits a strictly positive minimizer $u$. In certain 
cases we prove that this minimizer $u$ is strictly positive, but show by example that this need not be the case 
if the hypotheses are relaxed. 

Lest this criterion seem too abstract, observe that by conformal invariance, the local Yamabe invariant at a smooth point is 
equal to the Yamabe invariant of the round sphere; this is essentially what is known as Aubin's inequality. 
It is important that $Y_\ell$ involves the limits as $r \to 0$ of the Yamabe invariants $Y(B_r(p), g)$, rather than their 
values at any fixed $r > 0$; this means that local curvature invariants play a smaller role in $Y_\ell$. An invariant of 
this nature has been used previously for spaces $(M,g)$ with isolated conic singularities. In that setting, if $p$ is 
a conic point, so that some neighbourhood $\calU$ of $p$ in $M$ is modelled by a cone over a compact smooth 
Riemannian manifold $(Z,h)$, then the local Yamabe invariant at $p$ is the same as the so-called cylindrical Yamabe 
invariant $Y(\RR \times Z, [ dt^2 + h])$ which plays an important role in the work of the first author and Botvinnik 
\cite{AB}, see also \cite{Ak} for a discussion of this problem on orbifolds. It is proved there that $Q_g$ has a 
minimizer provided  
\begin{equation}
-\infty < Y(M,[g]) < \min_j \{ Y(S^n,[g_0]), Y( \RR \times Z_j, [dt^2 + h_j])\},
\label{conicineq}
\end{equation}
where $(C(Z_j), dx^2 + x^2 h_j)$, $j = 1, \ldots, N$, are the local models for the conic points of $(M,g)$. 
We note that $Y( \RR \times Z_j, [dt^2 + h_j]) \leq Y(S^n, [g_0])$ is always true. Also, implicit here is the 
fact that the cylinder $(\RR \times Z, dt^2 + h)$ and the cone $(C(Z), dx^2 + x^2 h)$ are 
mutually conformal.  We note that there are many examples of conic spaces where one does know that
\eqref{conicineq} holds, see also \cite{JR} for some results of the case where $R_g < 0$ (rather than
$Y(M,[g]) < 0$), when an existence result is obtained in some cases using barriers. 

Our main existence theorem states that if $(M,g,\mu)$ is an almost smooth metric-measure space which
satisfies the first three properties listed in the definition at the beginning of \S 1.1 below as well as one
of the three possible hypotheses on $\scal_g$, then $Q_{g,\mu}$ attains its minimum. In certain of these
cases, we also prove that the minimizing function $u$ is strictly positive on $M$.  The proof is divided 
into two parts: the proof of existence is obtained through a variant of the original method appearing in 
the work of Trudinger and Aubin, and the surprising fact is that this original proof may be 
adapted quite simply to this general setting. However, in order to accomodate some of the natural geometric 
applications later, we present an alternate proof of the step which uses Moser iteration to give a uniform 
upper bound for the minimizing sequence, by another argument related to some old ideas of Varopoulos. 
The proof that, in certain cases, the minimizer is strictly positive uses some ideas developed by Gursky. 

In the second part of this paper we expand on the theme and setting of \cite{AB} by considering in more detail 
the case where $(M,g)$ is a smoothly stratified Riemannian pseudomanifold, also known as an iterated edge 
space.  We identify the local Yamabe invariants at all point $p \in M$ as higher versions of the cylindrical/conic 
Yamabe invariants discussed above; these are simply the global Yamabe invariants for the model spaces 
$\RR^k \times C(Z)$, or (conformally) equivalently, $\HH^{k+1} \times Z$, where $Z$ is a compact iterated 
edge space with lower singular `depth' than the original space $M$.  The special cases of these invariants when
$Z = S^{n-k-1}$ play an interesting role in the work of Ammann, Dahl and Humbert \cite{ADH1}, \cite{ADH2},
where quantitative estimates of the change of the $\sigma$-Yamabe invariant (which is the supremum of the Yamabe
constants over all conformal classes) under surgeries are obtained.
Finally, using the more specialized analytic tools available for the study of PDE on smoothly stratified spaces, we 
prove sharp regularity results about the behaviour of the minimizer $u$ (or indeed any solution of the Yamabe equation)
at the singular strata of $M$. 

The final resolution of the Yamabe problem on smooth manifolds by Schoen, described in \cite{LP}, 
devolves to showing that $Y(M,[g]) < Y(S^n,[g_0])$ except when $(M,g)$ is conformal to $(S^n,g_0)$.  
One might hope for some analogue of this result here. For example, a natural conjecture is that if $(M,g)$ 
has only isolated conic singularities, then $Y(M,[g]) < Y_\ell(M,[g])$ unless $(M,g)$ is conformal to the cylinder 
$(\RR \times Z, dt^2 + h)$. Unfortunately, this is now known to be false! Indeed, a recent paper by Viaclovsky \cite{V} 
exhibits a manifold with orbifold singularity which does not admit any (incomplete) orbifold metric of constant scalar 
curvature, or equivalently, any finite energy critical point of $Q_g$.  Because the existence theory in \cite{AB} and 
\cite{Ak} would guarantee a minimizer unless the local and global Yamabe invariants are the same, we conclude that
these invariants must be equal for the spaces Viaclovsky considers.  There are no known examples beyond
the cylinder where $Y_\ell(M,[g]) = Y(M,[g])$ holds and $Q_g$ also has a critical point. It remains a tantalizing 
mystery to determine whether there is some rigidity phenomenon here. 

The authors acknowledge the following grant support: K.A. through the Grant-in-Aid for Scientific Research (B), 
JSPS, No. 24340008; G.C. through the ANR grant ACG: ANR-10-BLAN 0105; R.M. through the NSF grant DMS-1105050.
We are also grateful to Emmanuel Hebey for useful comments.

\section{The general existence theorem}
As discussed above, our main existence theorem yields a minimizer of the functional $Q_g$ on a rather broad class of 
Riemannian singular spaces. We state and prove this result in this section.  We first explain the precise geometric and 
analytic hypotheses, then define the local Yamabe and Sobolev invariants and describe their relationship to the 
(global) Sobolev constant of the space in question. We also give a number of auxiliary technical facts, including 
the compactness of the embedding $W^{1,2} \hookrightarrow L^{2p/(p-2)}$ for $p > n$, and that the finiteness of
the Sobolev constant implies discreteness of the spectrum of the (Friedrichs extension of the) Laplacian. 
We also review the standard Moser iteration argument to obtain a uniform upper bound for the subcritical
solutions and give a different proof based on a different (Morrey-type) assumption on the scalar curvature. 
Existence of the minimizer then follows the lines of the original Trudinger/Aubin argument. The positive 
lower bound for the minimizer uses an argument due to Gursky \cite{G}. 

\subsection{Almost smooth metric-measure spaces}
Suppose that $(M,d,\mu)$ is a compact metric-measure (MM) space which is `almost smooth' in the sense that 
there is an open dense subset $\Omega\subset M$ which is a smooth $n$-dimensional manifold, and a smooth Riemannian
metric $g$ on $\Omega$ which induces the same metric space structure as $d$ on $\Omega$, and hence by density
on all of $M$. (It is not hard to check that the arguments in this section only require that $g$ be $W^{2,q}$ for some $q>n/2$,
but for simplicity we do not work in this generality.)  We also assume that the measure $d\mu$ is a smooth positive
multiple of the volume form, i.e.\ $d\mu = h^2 dV_g$ for some $h \in \calC^\infty(\Omega) \cap \calC^0(M)$ which is
strictly positive.  Since $g$ and $d$ induce the same distance, we refer to the triple $(M,g,\mu)$ and omit mention of $d$.
Note that the metric balls $B(p,r)$ coincide with geodesic balls provided $B(p,r) \subset \Omega$.

We shall assume the following properties of the space $(M,g,\mu)$. 
\medskip

\noindent{\sl \boxed{\large{Hypotheses:}}} 
\begin{enumerate}
\item[i)] Let $W^{1,2}(M; d\mu)$ denote the Sobolev space which is the completion of the space of Lipschitz function $\Lip(M)$ with 
respect to the usual norm; then we assume that $\mcC^1_0(\Omega)$ is dense in $W^{1,2}(M;d\mu)$. Notice that this precludes
the existence of codimension one boundaries. 
\item[ii)] Hausdorff $n$-dimensional measure is absolutely continuous with respect to $d\mu$, and both of these measures 
are Ahlfors $n$-regular, i.e. 
\[
C^{-1}r^n \leq \mu( B(p,r)) \le C r^n
\]
for some $C > 0$, and for every $p \in M$ and $r \leq \mbox{diam}(M)$.  
Later, for simplicity, we often write $\mu(B(p,r)) = \Vol (p,r)$. 
\item[iii)] The Sobolev inequality holds: there exist $A, B > 0$ such that
\begin{equation}
\|f\|^2_{\frac{2n}{n-2}} \le A\int_{M}|df|^2 \, d\mu +B\int_{M}|f|^2 \, d\mu
\label{Sobolev}
\end{equation}
for all $f \in W^{1,2}(M; d\mu)$. 
\item[iv)] Finally, the scalar curvature function satisfies at least one of the following properties: 
\begin{itemize}
\item[a)] $\scal_g \in L^q(M, d\mu)$ for some $q > n/2$;
\item[b)] For some $q > 1$ there exists an $\alpha \in [0,2)$ such that for every point $p \in M$, 
\begin{equation}
\sup_{r > 0}  r^{-n} \int_{B(p,r)} |\scal_g|^q \, d\mu \leq C r^{-\alpha q}.
\label{Morrey}
\end{equation}
\item[c)] $\scal_g^- := \min\{\scal_g, 0\} \in L^q(M, d\mu)$ for some $q > n/2$;
\end{itemize} 
\end{enumerate}
We henceforth assume that conditions i) - iii) and at least one of iv) a)-c), are satisfied, unless explicitly stated 
otherwise. We call a triplet $(M,g,\mu)$ which satisfies these properties an almost smooth metric measure space.

The condition iv) b) states that $\scal_g$ lies in the Morrey regularity class $\calM^q_\lambda$ where $\lambda = n - \alpha q
\in (0,n]$. We state now an important fact about functions which lie in this class which will be used in several
places below. The proof is deferred to the end of \S 1.3 simply because it involves techniques which are discussed
there for other reasons, but does not rely on any of the intervening results. 
\begin{lem}
Suppose that the function $V$ satisfies \eqref{Morrey}. Then for any $\e > 0$ there exists $C_\e > 0$ such that
for all $\phi \in W^{1,2}(M)$, 
\[
\int |V| |\phi|^2 \, d\mu \leq \e \int |d\phi|^2\, d\mu + C_\e \int |\phi|^2 \, d\mu.
\]
\label{Morrey-lemma}
\end{lem}

\begin{rem}
The extra generality of allowing the measure $d\mu$ to be a smooth multiple of $dV_g$ rather than just the volume form itself,
plays very little role here. For simplicity in this paper, we usually assume that $d\mu = dV_g$.  Although the analysis 
in this paper goes through for more general measures, the conclusions then are no longer strictly within the realm
of conformal geometry. 
\end{rem}

\subsection{Yamabe and Sobolev constants}
For any open set $\calU \subset M$, we define its Sobolev and Yamabe constants, 
\[
\begin{array}{rcl}
S(\calU) & = &\inf\, \{\int |d\varphi|^2\, d\mu:  \varphi\in W^{1,2}_0(\calU), \|\varphi\|_{\frac{2n}{n-2}}=1 \}, \ \mbox{and} \\
Y(\calU) & = & \inf \, \{\int (|d\varphi|^2+\frac{n-2}{4(n-1)}\scal_g \varphi^2)\, d\mu: \\ 
& & \qquad \qquad \qquad \varphi\in W^{1,2}_0(\calU\cap\Omega), \|\varphi\|_{\frac{2n}{n-2}}=1 \},
\end{array}
\]
respectively. We also define the {\it local} Sobolev constant and {\it local} Yamabe invariant of $(M,g,\mu)$ by
\[
S_\ell(M,g) = \inf_{p\in M}\lim_{r\to 0} S(B(p,r)), \quad Y_\ell(M,[g]) = \inf_{p\in M}\lim_{r\to 0} Y(B(p,r)).
\]
All these quantities depend on $g$ and $\mu$, but we often suppress this, and even explicit mention of $M$,  in the notation. 

\begin{lem}
If $\scal_g$ satisfies either iv) a) or iv) b), then $Y_\ell(M,[g]) = S_\ell(M,g)$. 
\end{lem}
\begin{proof}
Assume first that iv) a) holds, i.e.\ that $\scal_g \in L^q$ for some $q > n/2$. By the H\"older inequality, 
\[
\left |Y(\calU)-S(\calU)\right|\le \frac{n-2}{4(n-1)}\|\scal_g\|_{q}\,\vol(\calU)^{\frac{2q-n}{nq}},
\]
and thus, since $q > n/2$, 
\[
\inf_{p\in M}\lim_{r\to 0} S(B(p,r)) = \inf_{p\in M}\lim_{r\to 0} Y(B(p,r))
\]
for any $p$. 

For the other case, we invoke Lemma~\ref{Morrey-lemma} as follows. Fix any $\e > 0$ and choose $C_\e$ accordingly. Then 
\begin{multline*}
\left| Y(B(p,r) - S( B(p,r))\right| \leq \inf_{\phi} \int_{B(p,r)} \frac{n-2}{4(n-1)} |\scal_g| \, |\phi|^2 \\
\leq \e S(B(p,r)) + C_\e \inf_{\phi} \int_{B(p,r)} |\phi|^2 \leq \e S(B(p,r)) + C_\e r^2,
\end{multline*}
where the H\"older inequality and the normalization $||\phi||_{2n/(n-2)} = 1$ are used to get the last term. Letting $r \searrow 0$
and taking the infimum over all $p \in M$ shows that $|Y_\ell- S_\ell| \leq \e S_\ell$, and since this
is true for all $\e > 0$, we see that $Y_\ell(M) = S_\ell(M)$, as claimed. 
\end{proof}

The fact that, under these hypotheses, $Y_\ell(M)$ is the same as the local Sobolev constant $S_\ell(M)$,
leads to an important criterion for the positivity of the local Yamabe invariant.
\begin{prop}\label{presquesobolev}
Let $(M,g)$ satisfy hypotheses i) - iii). 
\begin{itemize}
\item[a)] For any $\e > 0$ there exists $C_\e > 0$ such that
\begin{equation}
(S_\ell-\epsilon) \|f\|^2_{\frac{2n}{n-2}}\le \int_{\Omega}|df|^2\, d\mu +C_\epsilon\int_{\Omega}|f|^2\, d\mu
\label{Sob3}
\end{equation}
for all $f \in W^{1,2}(M)$.
\item[b)] If $\scal_g$ satisfies either iv) a) or iv) b), then $Y_\ell(M) > 0$ if and only if the Sobolev inequality 
\eqref{Sobolev} holds on $(M,g,\mu)$.  If these conditions are true, then for any
$\e > 0$ there is a constant $C_\epsilon > 0$ such that 
\begin{equation}
(Y_\ell-\epsilon) \|f\|^2_{\frac{2n}{n-2}}\le \int_{\Omega}|df|^2\, d\mu +C_\epsilon\int_{\Omega}|f|^2\, d\mu
\label{Sob2}
\end{equation}
for all $f \in W^{1,2}(M)$.
\end{itemize}
\end{prop}
\begin{proof}  Let us first address b). If \eqref{Sobolev} holds, then rearranging and using the H\"older inequality, 
we see that 
\[
S(\calU)\ge\frac{1}{A}\left(1-B(\vol \calU)^{2/n}\right),
\]
which implies a positive lower bound for $Y_\ell(M)$. 

On the other hand, suppose $Y_\ell(M) >0$. Fixing any $\delta\in (0,1)$, for each $p\in M$ there is a 
radius $r_p>0$ such that 
\begin{equation}
\min\{ Y(B(p,r_p)), S(B(p,r_p)) \}\ge (1-\delta) Y_\ell(M). 
\label{lb}
\end{equation}
Since $M$ is compact, there is a finite covering $M=\bigcup_i B(p_i, r_i)$ with $r_i=\frac12 r_{p_i}$. 
Hence if $s=\min r_i$, then \eqref{lb} is true for every $p \in M$ with $r_p$ replaced by $s$. 

Now choose a partition of unity $\{\rho_i\}$ subordinate to this covering such that each $\sqrt{\rho_i} \in \Lip(M)$. 
If $f\in W^{1,2}(M)$ and $\chi \in \Lip(M)$, then $\chi f\in W^{1,2}(M)$. Note that this follows from 
the density assumption i) and the fact that it holds on $\Omega$. Thus for $f\in W^{1,2}(M)$, 
\begin{multline*} 
(1-\delta) Y_\ell(M) \|f\|^2_{\frac{2n}{n-2}} = (1-\delta) Y_\ell (M) \|f^2\|_{\frac{n}{n-2}} \\
\le (1-\delta) Y_\ell(M)\sum_i \|\rho_i f^2\|_{\frac{n}{n-2}} \le \sum_i \|d(\sqrt{\rho_i} f)\|^2_{2} \\
\le  \sum \left( (1+\e')\|\sqrt{\rho_i} d f\|^2_{2}+  C_{\e'}\|f d\sqrt{\rho_i}\|^2_{2}\right) = (1+\e')||df||_2^2 + C ||f||_2^2
\end{multline*}
where $C$ depends on the $\rho_i$ and on $C_{\e'}$. This gives \eqref{Sob2} with $\epsilon = 2\delta Y_\ell(M)$ if
we choose $\e'$ appropriately.

The proof of a) is the same.
\end{proof}

\begin{rem}
Using the H\"older inequality and Lemma~\ref{Morrey-lemma}, we also deduce from any one of the 
hypotheses iv) a), b) or c) that
\[
Y(M,[g]) > -\infty.
\]
\label{rem-Yfinite}
\end{rem}

The Sobolev inequality has another important consequence. 
\begin{prop} If the Sobolev inequality holds on $(M,g)$, then the inclusion
\[
W^{1,2}(M) \longrightarrow L^{\frac{2p}{p-2}}(M)
\]
is compact for any $p \in (n,\infty)$. 
\end{prop}
\begin{proof}
We first write
\begin{multline*}
u-e^{-t(-\Delta+1)}u= - \int_0^t \frac{d\,}{ds} e^{-s(-\Delta+1)}u\, ds = \int_0^t e^{-s(-\Delta+1)}(-\Delta+1) u\, ds \\
= \int_0^t e^{-\frac{1}{2}s(-\Delta+1)}(-\Delta+1)^{\frac{1}{2}}e^{-\frac{1}{2}s(-\Delta+1)}(-\Delta+1)^{\frac{1}{2}} u\, ds.
\end{multline*}
We shall show that the $L^{2p/(p-2)}$ norm of this difference is bounded by $t^\beta ||u||_{1,2}$ for some $\beta > 0$.
This proves that inclusion mapping is approximated in the operator norm topology by a sequence of
compact mappings, which implies that it must be compact.

To do this, we need three facts.  First, if $u\in W^{1,2}(M)$, then
\[
\int (|du|^2+|u|^2) =\langle (-\Delta+1)u,u\rangle=\|(-\Delta+1)^{\frac{1}{2}} u\|^2_{2}.
\]
Next, it is an easy consequence of the spectral theorem that
\[
\left\|\sqrt{-\Delta + 1} e^{-s(-\Delta+1)}\right\|_{L^2 \to L^2} \leq C s^{-\frac12}.
\]
Finally, we claim that, if $q = 2p/(p-2) < 2n/(n-2)$, then
\[
\left\| e^{-s(-\Delta+1)}\right\|_{L^2\to L^q}\le \frac{C}{s^{\frac{n}{2}\left(\frac{1}{2}-\frac{1}{q}\right)}} = \frac{C}{s^{n/2p}},
\quad s > 0.
\]
This is proved by interpolation as follows. There is a standard estimate that
\[
\left \| e^{-s(-\Delta+1)} \right\|_{L^r \to L^r} \leq 1, 
\]
uniformly in $s$ for any $r \geq 2$. This follows by a simpler interpolation from the case $r=2$ (spectral theorem) and
$r=\infty$ (easy direct argument). We can obtain the same estimate for $1 < r < 2$ either by duality or noting that 
this also holds for $r=1$ and interpolating again. On the other hand, it is known \cite{SC1} that 
the Sobolev inequality \eqref{Sob2} implies that 
\[
\left\| e^{-s(-\Delta+1)}\right\|_{L^1 \to L^\infty}\le \frac{C}{s^{n/2}}, \quad s > 0.
\]
Thus if we interpolate between this $L^1 \to L^\infty$ estimate and the $L^r \to L^r$ estimate with 
$r = 1 - 1/p$, then a bit of arithmetic proves the claim. 

Putting these three estimates together, we conclude that
\[
\left\|u-e^{-t(-\Delta+1)}u\right\|_{q}\le C \int_0^t s^{-\frac12 - \frac{n}{2p}}\, ds \, ||u||_{1,2} 
= C t^{\frac12 - \frac{n}{2p}} \, ||u||_{1,2},
\]
which decays as required. 
\end{proof}

\begin{cor} 
Let $-\Delta$ be the self-adjoint operator obtained as the Friedrichs extension from the semi-bounded quadratic form 
\[
\langle \nabla f, \nabla f \rangle  = \int_M |\nabla f|^2 
\]
over the core domain $\calC^\infty_0(\Omega)$.  Then $-\Delta$ has discrete spectrum.
\label{cordiscspec}
\end{cor}
\begin{proof}
It suffices to show that the Friedrichs domain of $-\Delta$ is compactly contained in 
$L^2(M; d\mu)$.  However, this domain is simply $W^{2,2}(M; d\mu) \cap W^{1,2}_0(M;d\mu) 
\subset W^{1,2}(M; d\mu)$, which by the previous result is compactly contained in $L^{\frac{2p}{p-2}}$ for
any $p > n$, which in turn continuously includes in $L^2$. 
\end{proof}

\subsection{Uniform boundedness of subsolutions}
Let $(M,g,d\mu)$ be an almost smooth metric measure space, as considered above.  We now present two 
different methods which lead to the uniform boundedness of nonnegative functions $u$ which satisfy 
$\Delta u \geq Vu$. The first is simply the adaptation of the Moser iteration method to this setting, and 
assumes that $V$ satisfies either hypothesis iv) a), or c). This does not cover all the geometric cases 
we wish to consider, so we then prove a stronger result assuming that $V$ satisfies iv) b). 
This second result subsumes the first one, but we describe both proofs since the former is the more 
traditional method and certain constructions in its proof will be used later. 

\subsubsection{Moser iteration}
We  now review the classical Moser iteration method with enough detail to make clear that all steps work 
on almost smooth MM spaces. (In fact, Moser iteration works in greater generality still, see \cite{Mar}.)
\begin{prop}\label{moser}
Let $u\in W^{1,2}(M)$ be nonnegative and satisfy $\Delta u - Vu \ge 0$, where $V\in L^q$ 
for some $q>n/2$. Then $u\in L^\infty$ and
\[
\|u\|_{\infty} \leq C( \|V\|_{q}) \|u\|_{2},
\]
where the constant $C$ depends only on $n, q$, $||V||_q$ and the constants $A, B$ from the Sobolev inequality.
\end{prop}

\begin{rem}
As usual, the differential inequality is to be interpreted weakly, i.e. 
\begin{equation}
\label{weak}
\int (du,d\varphi)\le - \int Vu\varphi.
\end{equation}
for any $\varphi\in W^{1,2}(M)$ with $\varphi\ge 0$. Notice that the right hand side of \eqref{weak} is 
well defined because $V\in L^{n/2}$ and, from the Sobolev inequality, $u\varphi\in L^{n/(n-2)}$.
\end{rem}
\begin{proof} We follow the standard proof \cite[Theorem 8.15]{GT} as soon as we verify
the chain rule:

\medskip

\noindent{\bf Claim:} If $v\in W^{1,2}(M)$ and $f\in \mcC^1(\R,\R)$ satisfies $f'\in L^\infty$ then
$f\circ u\in W^{1,2}(M)$ and 
\[
d(f\circ u)=f'\circ u\,\,.\, du.
\]
\medskip

To prove this claim, note that by \cite[Theorem 7.5]{GT}, we have 
\[
\int (df\circ u, d\varphi)=\int (d u, d\varphi)\,f'\circ u \quad \mbox{for all}\ \varphi\in \mcC^1_0(\Omega),
\]
and the result follow from the density of $\mcC^1_0(\Omega)$ in $W^{1,2}.$

\medskip

For $\alpha\ge 2$, define
\begin{equation}\label{fa}
f_\alpha(x) = 
\begin{cases}
x^\alpha &{\ \rm if\ } 0\le x\le \alpha^{- \frac{1}{\alpha-1}} \\ 
x + (\alpha^{-\frac{\alpha}{\alpha-1}} - \alpha^{-\frac{1}{\alpha-1}}) &{\ \rm if\ }  \alpha^{-\frac{1}{\alpha-1}}\le x;
\end{cases}
\end{equation}
the cutoff and additive constant are chosen so that $f_\alpha(x)$ is $\mcC^1$ and convex. Next, for any $L \geq 1$, let
\[
\phi_{\alpha,L}(x)=L^\alpha f_\alpha\left(\frac{x}{L}\right).
\]
Note that $\phi_{\alpha,L}(x) = x^\alpha$ on larger and larger intervals as $L \to \infty$. Furthermore, if we define
$G_{\alpha,L}(x)=\int_0^x \phi_{\alpha,L}'(t)^2 \, dt$, then a laborious computation gives
\begin{equation}
\phi_{\alpha,L}(x)\le x^\alpha \ \mbox{and}\quad xG_{\alpha,L}(x)\le\frac{\alpha^2}{2\alpha-1}  \left(\phi_{\alpha,L}(x)\right)^2,
\label{ineqG}
\end{equation}
for all $x \geq 0$. 

Inserting $\varphi=G_{\alpha,L}(u)$ into (\ref{weak}) and using \eqref{ineqG}, we obtain
\begin{multline*}
\int | d\phi_{\alpha,L}(u)|^2 =\int G_{\alpha,L}'(u)|du|^2 =\int (du,d\varphi)\\
\le \int Vu\, G_{\alpha,L}(u) \le \frac{\alpha^2}{2\alpha-1}\int V\phi_{\alpha,L}(u)^2. 
\end{multline*} 
Use both the Sobolev and H\"older inequalities to get
\begin{equation*}
\begin{split}
\|\phi_{\alpha,L}(u)\|^2_{L^{\frac{2n}{n-2}}}&\le \int \left(\frac{\alpha^2}{2\alpha-1}AV+B\right)\left(\phi_{\alpha,L}(u)\right)^2\\
&\le \frac{\alpha^2}{2\alpha-1} C \left(\int \left(\phi_{\alpha,L}(u)\right)^{\frac{2q}{q-1}}\right)^{\frac{q-1}{q}},
\end{split}\end{equation*}
with $C=A\|V\|_{q}+B (\vol M)^{1/q}.$  
Letting $L \to \infty$ yields
\begin{equation}
\left( \int u^{\frac{2\alpha n}{n-2}} \right)^{\frac{n-2}{n}} \leq \frac{C \alpha^2}{2\alpha-1} 
\left( \int u^{\frac{2\alpha q}{q-1}} \right)^{\frac{q-1}{q}}.
\label{improve}
\end{equation}
This is of course only interesting if the right side is finite. 

We are given the initial choice of $q$ through the potential $V$, with $q > n/2$. Thus $r := 2q/(q-1) < 2n/(n-2)$ and 
we can choose $\alpha$ sufficiently close to $1$ so that $\alpha \, r < 2n/(n-2)$ as well. Since $W^{1,2} 
\hookrightarrow L^{\alpha r}$, the right hand side of \eqref{improve} is finite, and hence so is the left, i.e. $u \in 
L^{\kappa \alpha  r}$ where
$$
\kappa:=\frac{n}{n-2}\frac{q-1}{q}>1. 
$$ 
Furthermore, we can rewrite \eqref{improve} as 
\begin{equation}
||u||_{\kappa \alpha r} \leq (C_1 \alpha)^{\frac{1}{2\alpha}} ||u||_{\alpha r},
\label{improve2}
\end{equation}
where $C_1 = ( C \alpha/(2\alpha - 1))$.  Note that \eqref{improve2} is valid for any $\alpha_j \geq \alpha$ so
long as $u \in L^{\alpha_j r}$, since $\alpha_j/(2\alpha_j-1)$ is uniformly bounded.  Now set $\alpha_j = \kappa^j \alpha$,
and apply \eqref{improve2} inductively to obtain that
\[
||u||_{\kappa^N \alpha r} \leq \left(\prod_{j=0}^{N-1} \left( C_1 \kappa^j \alpha\right)^{\frac{1}{2 \kappa^j}}\right)||u||_{\alpha r}
\]
for any $N \geq 1$. Finally, note that the constant here is bounded independently of $N$; indeed
\[
\log \prod_{j=0}^{N-1} \left( C_1 \kappa^j \alpha\right)^{\frac{1}{2 \kappa^j}} = \sum_{j=0}^{N-1} 
\left(\frac{ \log C_1 \alpha}{2\kappa^j} + \frac{\log \kappa^j}{2\kappa^j}\right) \leq C_2 \frac{\kappa-1}{\kappa}. 
\]
Thus, taking the limit as $N \to \infty$, we obtain finally that
\[
\|u\|_{\infty}\le C(q,n,A,B, ||V||_q) \|u\|_{1,2}.
\]
\end{proof}

\subsubsection{Varopoulos' method}
For the second method, we now suppose that $V$ satisfies the Morrey condition iv) b) for some $q > 1$ and 
$0 \leq \alpha < 2$. 
To compare this hypothesis with the hypothesis in Proposition~\ref{moser}, observe simply that by
the H\"older inequality, if $V\in L^{n/\alpha}$, then \eqref{Morrey} holds (with $\scal_g$ replaced by $V$) 
provided $q < n/\alpha$.

Next, the existence of the Sobolev inequality \eqref{Sob2} implies the Gaussian upper bound
\begin{equation}
e^{t\Delta}(x,y)\le C\frac{1}{t^{n/2}}e^{-\frac{d(x,y)^2}{5t}} 
\label{Gaussian}
\end{equation}
for the Schwartz kernel of the heat operator $e^{t\Delta}$, see \cite{Coulhon}. Thus if $G(x,y)$ denotes the Green kernel associated 
to $-\Delta +1$, then
\[
0 < G(x,y)\le \frac{C}{d^{n-2}(x,y)}
\]
provided $d(x,y) \leq C$. 

Our goal, as before, is to prove the following. 
\begin{thm}
Assume that $u$ is a nonnegative function in $W^{1,2} \cap L^p$ for some $p>q^*$ such that 
\[
\Delta u \geq Vu
\]
where $V$ satisfies \eqref{Morrey}. Then 
\[
||u||_\infty \leq C \left(\sup_x \sup_{r > 0} \frac{r^{q\alpha} }{\Vol(x,r)}\int_{B(x,r)} |V|^q\, d\mu\right) \, ||u||_p . 
\]
(In other words, the constant $C$ depends in some possibly nonlinear way on the quantity in parentheses.)
\label{bddGV2}
\end{thm}

The key point is to rewrite the differential inequality as $(-\Delta + 1) u \leq (-V+1)u$ and then, using that 
$G$ is positivity preserving, $u \leq G \circ (-V+1) u$.  Clearly $-V+1$ satisfies \eqref{Morrey} if $V$ does,
so for simplicity we replace $-V+1$ by $V$. We then establish the following mapping properties.
\begin{thm} Let $L$, $G$ and $V$ be as above.  Then $G \circ V$ is bounded as a mapping
\[
\begin{split}
& L^p(M,d\mu)\rightarrow L^\infty(M,d\mu) \quad \mbox{when}\ p>\frac{n}{2-\alpha} \\
& L^p(M,d\mu)\rightarrow L^{\frac{pn}{n-(2-\alpha)p}}(M,d\mu), \quad \mbox{when}\ 
p\in\left(q^*\,,\, \frac{n}{2-\alpha}\right), \ \ (q^*=q/(q-1))
\end{split}
\]
\label{estvar}
\end{thm}
\begin{proof}
Define the Stieljes measure $d\nu_x(t)$ associated to the nondecreasing function 
\[
t\mapsto \nu_x(t)= \int_{B(x,t)}|V(y)|\,|f(y)| \, d\mu(y) \, ,
\]
and then write  
\begin{equation}
\label{formula}
\begin{split}
|(G\circ V)f(x)|& \le \int_M \frac{C}{d^{n-2}(x,y)}|V(y)|\,|f(y)|\, d\mu(y) \\ & 
\le\int_0^{D}\frac{C}{r^{n-2}}\, d\nu_x(r) \\ & =\frac{C}{D^{n-2}}\nu_x(D)+(n-2)\int_0^D 
\frac{C}{r^{n-1}}\nu_x(r)\, dr,
\end{split}
\end{equation}
where $D = \mbox{diam}\, M$. We write the right side of this chain of inequalities as
$T_0(f) + T_\infty(f)$, and prove the boundedness properties for these operators separately. 

The estimate of the first of these is trivial. Indeed, for any $p > q^*$, 
\[
\nu_x(D)\le C \|V\|_q\,\|f\|_p \Longrightarrow \|T_0(f)\|_\infty\le C\|f\|_p.
\]
We can thus concentrate on $T_\infty$.

By the H\"older inequality, 
\[
\nu_x(r)\le C  r^{-\alpha} \| f\|_p\, V(x,r)^{1-\frac{1}{p}}\le C \|f\|_p r^{n-\alpha-\frac{n}{p}}.
\]
Thus if $p>\frac{n}{2-\alpha}$, then
\[
\|T_\infty(f)\|_\infty = \sup_x \int_0^D \frac{C}{r^{n-1}}\nu_x(r)dr\le C' \|f\|_p \int_0^D r^{1-\alpha-\frac{n}{p}}dr\le C'' \|f\|_p\, .
\]

In the second case, when $q^* < p < \frac{n}{2-\alpha}$, this integral no longer converges near $0$, so we use 
instead a classical cutting argument  from harmonic analysis, replacing the estimate for $\nu_x(r)$ for $r$ small 
by a different one. 

Define
\[
v(x)=\mcM\left(|f|^{q^*}\right)^{\frac{1}{q^*}}(x),
\]
where $\mcM(f)$ is the maximal function, defined for any $L^1_{\mathrm{loc}}$ function $f$ by
\[
\mcM(f)(x):=\sup_{r>0}\frac{1}{V(x,r)}\int_{B(x,r)} |f|\, d\mu.
\]
Hypothesis ii), Ahlfors $n$-regularity, implies volume doubling, i.e.\ $\Vol (x,2r) \leq C \Vol (x,r)$ for all $x \in M$ 
and $r > 0$, and from this it is easy to deduce that the maximal function, is bounded $L^1 \to L^1_{\mathrm{weak}}$,
\[
\mu\left\{\mcM(f)>\lambda\right\}\le \frac{C}{\lambda}\,\|f\|_1.
\]
Since $\mcM$ is also (trivially) bounded $L^\infty\to L^\infty$, by interpolation we see that it is also bounded 
$L^s \to L^s$ for any $1 < s \leq \infty$. This will be invoked below.

From the definition of $v$, 
\[
\nu_x(r)\le C r^{\frac{n}{q}-\alpha} V(x,r)^{\frac{1}{q^*}}\, v(x)\le r^{n-\alpha}v(x),
\]
and hence 
\begin{equation*}\begin{split}
T_\infty(f)(x)&\le \int_0^{\lambda(x)}\frac{C}{r^{n-1}}\nu_x(r)dr+ \int_{\lambda(x)}^D\frac{C}{r^{n-1}}\nu_x(r)dr\\
&\le  \int_0^{\lambda(x)}\frac{C}{r^{n-1}}r^{n-\alpha}v(x)dr+\int_{\lambda(x)}^D\frac{C}{r^{n-1}} r^{n-\alpha-\frac{n}{p}}\, \| f\|_p dr\\
&\le C \lambda(x)^{2-\alpha} v(x)+C  \lambda(x)^{2-\alpha-\frac{n}{p}}\| f\|_p
\end{split}\end{equation*}
for any $0 < \lambda(x) < D$. The optimal choice of $\lambda(x)$ satisfies
\[
v(x)= \lambda(x)^{-\frac{n}{p}}\| f\|_p,
\]
so inserting this yields
\begin{equation}
\begin{split}
&T_\infty(f)(x)\le C \|f\|_p^{(2-\alpha)\frac{p}{n}}\, v(x)^{1-(2-\alpha)\frac{p}{n}}  \Longrightarrow \\
& \left|T_\infty(f)(x)\right|^{\frac{pn}{n-(2-\alpha)p}}\le C \|f\|_p^{\frac{(2-\alpha)p^2}{n-(2-\alpha)p}}\, v(x)^{p}.
\end{split}
\end{equation}
Finally, using the boundedness of the maximal function on $L^s$, $s = p/q^*$, we obtain 
\[
\|v\|_p \le \|f\|_p,
\]
whence, after some arithmetic, 
\[
\left\|T_\infty(f)\right\|_{\frac{pn}{n-(2-\alpha)p}}\le C\|f\|_p.
\]
This is the desired estimate.
\end{proof}

This result implies Theorem~\ref{bddGV2} quite directly. Indeed, reverting 
back to the original potential, we have already noted that $0 < u \leq G \circ (-V+1)u$ and $u \in L^p$.
If $p > n/(2-\alpha)$, then the first part of Theorem~\ref{estvar} bounds $||u||_\infty$ immediately.
On the other hand, if we only know that $p > q^*$, then the second part of this Theorem shows that
$G \circ (-V+1)u$ lies in $L^{p_1}$ where $p_1 = p \left( n/( n - (2-\alpha)p) \right)$. It is easy to check
that there exists $\e > 0$ such that $p_1/p \geq 1+\e$ for any $p > q^*$, which means that we can
iterate this procedure, obtaining successively that $u \in L^{p_j}$ for an increasing sequence $p_j$ with
$p_j \geq p (1+\e)^j$. Hence $p_N > n/(2-\alpha)$ for some $N$, so that at the next step $u \in L^\infty$. 

We conclude this section with the
\medskip

\noindent{\it Proof of Lemma~\ref{Morrey-lemma}}
We begin by noting that under the assumptions of this Lemma, the heat kernel bound \eqref{Gaussian}
holds, hence the Schwartz kernel $K_\mu(x,y)$ corresponding to 
\[
\left(-\Delta +1+\mu^2\right)^{-1/2} = C \int_0^\infty t^{-1/2} e^{-t (-\Delta+1 + \mu^2)}\, dt 
\]
satisfies
\begin{equation}\label{boundDelta12}
0 < K_\mu(x,y)\le C\frac{e^{-\delta \mu d(x,y)}}{d^{ n-1}(x,y)}\,\,.
\end{equation}
for some $\delta > 0$. Moreover, defining $\varphi_\mu(r):=Ce^{-\delta \mu r}/r^{ n-1}$, then 
\[
-\varphi'_\mu(r)\le C'\frac{e^{-\frac{\delta}{2} \mu r}}{r^{ n}}.
\]

In order to prove this Lemma, it is known (see \cite[Theorem X.18]{RSi}), that it suffices to show that the operator
\[
A_\mu:= \left(-\Delta+1+\mu^2\right)^{-1/2}|V|^{\frac 12}\colon L^2(X,\mu)\rightarrow  L^2(X,\mu),
\]
satisfies
\[
\lim_{\mu\to+\infty} \left\| A_\mu\right\|_{L^2\to L^2}=0.
\]

Define $q^* \in (1,2)$ by $\frac{1}{2q}+\frac{1}{q^*}=1$. Then for $u\in L^2(X,\mu)$, we set 
\[
v(x):=\left(\sup_{r>0}\frac{1}{\mu(B(x,r))}\int_{B(x,r)} |u|^{q^*}d\mu\right)^{\frac{1}{q^*}}\, .
\]
Because $q^*\in (1,2)$, we have that $\|v\|_2\le C \| u\|_2$.

Now introduce the Stieljes measure $\tilde{\nu}_x$ associated to the nondecreasing function
\[
r\mapsto \int_{B(x,r)} |V|^{\frac{1}{2}} |u|\, d\mu
\]
so that 
\[
\left|(A_\mu u)(x)\right|\le C\int_0^D \varphi_\mu(r)\, d\tilde{\nu}_x(r)=C\varphi_\mu(D)\tilde{\nu}_x(D)-
\int_0^D \varphi'_\mu(r)\tilde{\nu}_x(r)\, dr,
\]
$D = \mbox{diam}\, M$.  Using the H\"older inequality and the fact that $|V|^{\frac{1}{2}}\in L^{2}$ (since $V \in L^q$
for some $q > 1$), we get
\[
\tilde{\nu}_x(D)\le C\|u\|_2, \qquad \mbox{and} \qquad \tilde{\nu}_x(r)\le C r^{n-\alpha/2} v(x).
\]
This gives 
\begin{equation*}\begin{split}
\left|(A_\mu u)(x)\right|&\le C\varphi_\mu(D)\|u\|_2+\int_0^D e^{-\frac{\delta}{2} \mu r} r^{-\alpha/2}dr \,v(x)\\
&\le C\varphi_\mu(D)\|u\|_2+C\mu^{1 - \frac{\alpha}{2}} v(x)\, ,
\end{split}
\end{equation*}
which proves finally that
\[
\left\| A_\mu\right\|_{L^2\to L^2}\le C\varphi_\mu(D)+C\mu^{1 - \frac{\alpha}{2}}
\]
Since $\alpha<2$, we deduce the result.   \hfill $\Box$

\subsection{Existence of  the minimizer}
We are now in a position to prove the basic existence result. 

\begin{thm}
Let $(M,g,\mu)$ be a compact almost smooth metric measure space, in particular satisfying hypotheses i) - iii) in \S 1.1,
and such that $\scal_g$ satisfies either iv) a),  b) or c).  Supposing that $\scal_g$ satisfies iv) a) or b), we then assume
\[
Y(M,[g]) < Y_\ell(M,[g]),
\]
while if $\scal_g$ satisfies iv) c), then our assumption becomes 
\[
Y(M,g) < S_\ell(M,g).
\]
Then there exists a function $u\in W^{1,2}(M)\cap L^\infty(M)$ such that $\|u\|_{\frac{2n}{n-2}}=1$ and
\[
Y(M,[g])=\int \left(|du|^2+\frac{n-2}{4(n-1)}\scal_g u^2\right)\, d\mu.
\]
Hence on the smooth locus $\Omega \subset M$,
\[
\Delta u - \frac{n-2}{4(n-1)}R_g u + \frac{n-2}{4(n-1)} Y(M,[g],\mu) u^{\frac{n+2}{n-2}} = 0.
\]
\label{existence}
\end{thm}
\begin{proof}  We follow the lines of the classical proof of Trudinger and Aubin. 
Since $W^{1,2}(M)\hookrightarrow L^{2p/(p-2)}$ is compact when $p>n$, the minimum value 
\[
Y_p= \inf\left\{\int \left(|d\varphi|^2+\frac{n-2}{4(n-1)}\scal_g \varphi^2\right)\, d\mu \right\}
\]
over all $\varphi\in W^{1,2}(M;d\mu)$ with $\|\varphi\|^2_{\frac{2p}{p-2}}=1$ is attained by some function $u_p$. 
The usual arguments from the calculus of variations show that 
$u_p\ge 0$ and 
\begin{equation}
\Delta u_p - \frac{n-2}{4(n-1)}\scal_g u_p + Y_p \, u_p^{\frac{p+2}{p-2}} = 0.
\label{pEL}
\end{equation}
It follows from Theorem~\ref{bddGV2} that $u_p\in L^\infty$. Indeed, \eqref{pEL} implies that $\Delta u_p \geq Vu_p$, 
where $V = c_n \scal_g^- - Y_p u^{4/(p-2)}$. Under any of the hypotheses a)-c), $\scal_g^-$ satisfies \eqref{Morrey}; 
on the other hand, setting $s = p/2q$, we have
\[
r^{-n}\int_{B(x,r)} |u_p^{\frac{4}{p-2}}|^q \, d\mu \leq ||u_p||_{\frac{2p}{p-2}}^{1/s} \, r^{n/s^* - n} = r^{-\alpha q}
\]
where $\alpha = 2n/p < 2$, so the second summand in $V$ satisfies \eqref{Morrey} as well. Hence we may apply
this theorem as claimed. 

Now, $\lim_{p\to n} Y_p=Y(M) < Y_\ell$, so for any sufficiently small $\e > 0$, $Y_p\le Y_\ell-\epsilon$ provided $p$  
is sufficiently close to $n$. We now may as well replace $Y_\ell(M)$ by $S_\ell(M)$, and argue assuming that 
$\scal_g^- \in L^q$ for some $q > n/2$. 

Since $||u_p||_{1,2}$ is uniformly bounded, we can choose a subsequence $p_j \to n$ such that $u_{p_j}$ converges to some function 
$u$, weakly in $W^{1,2}$ and strongly in $L^q$ for all $q\in [1,2n/(n-2) )$.  Our goal is to show that some further subsequence
converges strongly in $L^{\frac{2n}{n-2}}$. For if this is the case, then we can pass to the limit in 
\begin{equation}
\int \langle du_p , d \varphi\rangle+\frac{n-2}{4(n-1)}\scal_g u_p\, \varphi=Y_p \int \varphi \, u_p^{\frac{p+2}{p-2}} 
\label{weakp}
\end{equation}
to conclude that $u \in W^{1,2} \cap L^\infty$ is a weak, and hence strong, solution of the equation with $||u||_{\frac{2n}{n-2}} = 1$; 
setting $\varphi = u$ into \eqref{weakp} with $p=n$ then gives that $Q_g(u) = Y(M, [g])$, as desired. 

To accomplish this, recall the function $G_{\alpha, L}$ introduced in \S 1.3.1; inserting $\varphi =G_{\alpha,L}(u_p)$ 
with $\alpha\simeq 1$ into \eqref{weakp}, and using the properties of these functions, gives
\begin{equation}
\int  |d \phi_{\alpha,L}(u_p)|^2 \leq c(n) \int |\scal_g^-| |\phi_{\alpha,L}(u_p)|^2+ Y_p^+ \int u_p^{\frac{4}{p+2}} |\phi_{\alpha,L}(u_p)|^2, 
\label{afterPoinc}
\end{equation}
where $Y^+_p=\max\{ Y_p, 0\}$ and $c(n) = (n-2)/4(n-1)$. 
Next, using this in the Sobolev inequality \eqref{Sob2}, with $\epsilon$ replaced by 
$\epsilon/2$, yields that
\begin{multline*}
(S_{\ell} - \epsilon/2) ||\phi_{\alpha,L}(u_p)||_{\frac{2n}{n-2}}^2 \leq \int (C_{\epsilon/2} + c(n) |\scal_g^-|)\, |\phi_{\alpha,L}(u_p)|^2 \\
+ \frac{\alpha^2}{2\alpha-1}Y_p^+ \int u_p^{\frac{4}{p+2}} |\phi_{\alpha,L}(u_p)|^2.
\end{multline*}
By the H\"older inequality and the normalization of $u_p$,
\begin{multline*}
\int \phi_{\alpha,L}(u_p)^2u_p^{\frac{4}{p-2}}\le \|\phi_{\alpha,L}(u_p)\|^2_{\frac{2n}{n-2}}\|u_p\|_{\frac{2p}{p-2}}^{2/p}(\mathrm{Vol} \, M)^{2(p-n)/pn}
\\ = \|\phi_{\alpha,L}(u_p)\|^2_{\frac{2n}{n-2}}(\vol M)^{2(p-n)/pn}.
\end{multline*}
Furthermore, if $|\alpha-1|$ and $|p-n|$ are both sufficiently small, then 
\[
\frac{\alpha^2}{2\alpha-1} (\mathrm{Vol}\, M)^{2(p-n)/pn} \leq 1+\epsilon'. 
\]
Now choose $\epsilon'$ so that $(1+\epsilon') Y_p^+ \leq S_\ell - 3\epsilon/4$. Rearranging the inequality above, we obtain
\[
\frac{\epsilon}{4} \|\phi_{\alpha,L}(u_p)\|^2_{\frac{2n}{n-2}}\le C ||\phi_{\alpha,L}(u_p)||_2^2 + c(n) \int |\scal_g^-| \, |\phi_{\alpha,L}(u_p)|^2.
\]

We handle this last term in two different ways, depending on whether $\scal_g$ satisfies iv)a) or c), or else iv) b). 
In the former cases, the H\"older inequality estimates this term by $||\scal_g^-||_q ||\phi_{\alpha,L}(u_p)||_{2q/(q-1)}^2$.
Since $2\alpha < 2p/(p-2)$ and $\alpha q/(q-1) < 2p/(p-2)$ uniformly as $p \searrow n$, we can then pass to a limit
as $L \to \infty$ to conclude that 
\[
\|u_p\|_{\frac{2n\alpha}{n-2}}\le C
\]
for some $C$ which is independent of $p$.   If $\scal_g$ satisfies iv) b), then we use Lemma~\ref{Morrey-lemma}
already in \eqref{afterPoinc} to absorb this term at the expense of an extra factor of $(1+\e')$ in front
of the $Y_p^+$, but we can then proceed exactly as before to reach the same conclusion.

To conclude, observe that by Proposition (\ref{moser}), and using \eqref{weakp}, we obtain that $\|u_p\|_{\infty}\le C$ with 
$C$ independent of $p$. This leads immediately to the strong convergence of $u_p$ to $u$ in $L^{2n/(n-2)}$. 
\end{proof}

\begin{rem}
The existence of a minimizer $u\in W^{1,2}$ can also be proved using the almost optimal Sobolev inequality and a 
useful trick of Brezis and Lieb \cite{BrezisAndLieb}. However, it still requires the same amount of work as above to
prove that this minimizer is bounded.
\end{rem}

\subsection{The lower bound}
We now show that when $\scal_g$ satisfies either iv) a) or b), then the minimizer obtained in the last subsection is 
strictly positive. As we show by explicit example later, this fails when $\scal_g$ only satisfies iv) c). This lower bound
is attained by adapting an argument of Gursky \cite[Lemma 4.1]{G}. 

\begin{lem} For any ball $B(p,r)$ properly contained in $M$, there is a constant $C > 0$ such that
\[
\|\varphi\|_{L^2(M)}\le C\left(\|d \varphi\|_{L^2(M)}+\|\varphi\|_{L^2(B)}\right) \quad \forall\, \varphi\in W^{1,2}(M).
\] 
\end{lem}
\begin{proof} 
Let $B'=\frac12 B$ be the ball with the same center and half the radius. Since $\vol B' > 0$, there
is a Poincar\'e inequality on the complement of $B'$,
\[
\|\varphi\|_{2}\le C\|d\varphi\|_{L^2}. \quad \forall \varphi\in W^{1,2}_0(M\setminus B').
\]
Hence if $\rho(x)$ is a Lipschitz cutoff function which equals $1$ in $B'$ and vanishes outside $B$, then
\begin{equation*}\begin{split}
\|\varphi\|_{L^2(M)}&\le \|\varphi\|_{L^2(B)}+\|(1-\rho)\varphi\|_{L^2(M)}\\
&\le  \|\varphi\|_{L^2(B)}+C\|d [(1-\rho)\varphi]\|_{L^2(M \setminus B')}\\
&\le C' \left( || d\varphi ||_{L^2(M)} + ||\varphi||_{L^2(B)}\right)
\end{split}\end{equation*}
as claimed. 
\end{proof}

\begin{prop}
Assume that $\scal_g$ satisfies either iv) a) or b). Let $u$ be the minimizing solution obtained in the 
last subsection.  Then $\inf_M u > 0$. 
\label{lowerbound}
\end{prop}
\begin{proof}
We know that for every $\varphi\in W^{1,2}(M)$, 
\[
\int_M\langle du,d\varphi\rangle +\int_M \frac{n-2}{4(n-1)}\scal_g u\varphi=Y\int_M u^{\frac{n+2}{n-2}}\varphi.
\]
It is easy to check that for $\epsilon,\delta>0$ the function 
\[
\varphi=(\epsilon+u)^{-1-2\delta}\in W^{1,2}.
\]
Inserting this into the identity above gives 
\[
-\frac{1+2\delta}{\delta^2}\int_M \left| d(\epsilon+u)^{-\delta}\right|^2
=-\int_M \frac{n-2}{4(n-1)}\scal_g u\varphi+Y\int_M u^{\frac{n+2}{n-2}}\varphi,
\]
and hence
\begin{equation}
\int_M \left| d(\epsilon+u)^{-\delta}\right|^2\le 
\frac{\delta^2}{1+2\delta}C\left[  \int_M (|\scal_g|+
u^{\frac{4}{n-2}}) (\epsilon+u)^{-2\delta}\right]
\label{ce}
\end{equation}
with $C=\max\left \{ \frac{n-2}{4(n-1)}, |Y|\right\}$. By the H\"older inequality, and assuming
that $\scal_g$ satisfies iv) a), 
\[
\int_M \left| d(\epsilon+u)^{-\delta}\right|^2\le  \frac{\delta^2}{1+2\delta}C'  
\left[ \int_M (\epsilon+u)^{-\frac{2\delta n}{n-2}}\right]^{1-\frac 2n},
\]
where $C'=C\,\left[ \|\scal_g\|_{L^{n/2}}+ \|u\|_{L^{2n/(n-2)}} \right]$. If $\scal_g$ satisfies iv) b) instead,
then we handle the first term on the right using Lemma~\ref{Morrey-lemma} in an obvious way, and
end up with the same inequality. 

Applying the Sobolev inequality to the function $(u+\e)^{-\delta}$ and using the Lemma above
and this inequality, we conclude that for $\delta$ small enough, 
\[
\left( \int_M (\epsilon+u)^{-\frac{2\delta n}{n-2}}\right)^{1-\frac 2n} \leq C \int_B (\epsilon+u)^{-2\delta}.
\]
Assuming that $\overline{B} \subset \Omega$, then by the known upper bound on $u$ and the Harnack
inequality, there is a $c > 0$ such that $u\ge c>0$ on $B$. Letting $\epsilon\to 0$ in the estimate above, we 
see that
\[ 
\int_M u^{-\frac{2\delta n}{n-2}} \le C. 
\]

To conclude the proof, note that the convexity of $x\mapsto x^{\delta}$ implies that 
\[
\Delta (u^{-\delta})\geq \left(-\delta \frac{n-2}{4(n-1)}\scal_g + \delta Y u^{\frac{4}{n-2}}\right) u^{-\delta} := V u^{-\delta}
\]
This function $V$ satisfies \eqref{Morrey}, so Theorem~\ref{bddGV2} gives that $||u^{-\delta}||_\infty < \infty$, i.e.\ $\inf u  > 0$. 
\end{proof}

\section{The Yamabe problem on stratified spaces}
We now specialize the results of the last section to the setting of spaces with smooth stratifications, also called
iterated edge spaces, with corresponding adapted iterated edge metrics. We begin by reviewing some aspects of
the differential topology and metric structure of these spaces, then prove that all such spaces satisfy a
Sobolev inequality and the other hypotheses from \S 1.1, and then formulate the precise existence theorem in
this setting.

\subsection{Smoothly stratified spaces}
We now briefly review the definition of smoothly stratified pseudomanifolds. Further details can be found in in 
the foundational monograph of Verona \cite{Ve} and the exposition by Pflaum \cite{Pfl}.  Basic definitions 
vary between sources, and the recent paper \cite{ALMP} provides a clarification and unified presentation of some 
of this material; we follow the notation and development of \cite[\S 2]{ALMP} and refer to it for all further
details, in particular, for a proof that this class of spaces coincides with the class of iterated edge spaces 
considered by Cheeger \cite{Ch}, cf. also \cite{Maz-evian}.  

Let $X$ be a compact stratified space.  By definition, $X$ admits a disjoint decomposition into strata, $X = 
\sqcup \Sigma_j$, where each $\Sigma_j$ is a (possibly disconnected, possibly open) manifold of dimension $j$. 
There are a set of axioms describing how the strata fit together, key amongst which is that each connected 
component of $\Sigma_j$ has a tubular neighbourhood $\calU$ which is the total space of a smooth bundle over 
that component with fibre a truncated cone $C(Z_j)$. Here $Z_j$ is itself a compact stratified space and is 
called the link of that cone bundle.  There is a natural filtration of $X$ in terms of `depth' of singularities. Thus compact 
smooth manifolds are said to have depth $0$, and if $Z$ is a compact space of depth $k$, then a space which has
a neighbourhood which is a truncated cone or a bundle of truncated cones with link $Z$ has depth $k+1$. This depth filtration is 
different than the filtration of $X$ determined by the closures $\overline{\Sigma_j}$ since a stratum of high 
codimension can have low depth (for example, an isolated conic singularity only has depth $1$).  An interesting 
subtlety is that the fibration of each tubular neighbourhood is required to have a smooth trivialization, but it is not
a priori obvious what the proper class of smooth maps and diffeomorphisms between stratified spaces should be.
This is precisely the point where the various treatments cited above differ. The definition we give is inductive: once a 
suitable definition of a stratified diffeomorphism between spaces of depth $j$ has been given, one declares the 
suspension of such a diffeomorphism, i.e.\ the radial extension of that diffeomorphism to the cone over that space, 
to be smooth. This extends the definition to spaces of depth $j+1$. (This smoothness hypothesis excludes many 
spaces that constitute the standard broader class of stratified spaces, where these local trivializations are only 
required to be continuous; see \cite{ALMP} for more on this.) A smoothly stratified space $X$ is a pseudomanifold 
if the stratum of maximal 
dimension is dense in $X$. In distinction to \cite{ALMP}, we allow $X$ to have strata of codimension one, but 
since hypersurface boundaries play a somewhat different role in our main theorem, we say that $X$ is an iterated 
edge space (or smoothly stratified pseudomanifold) with boundary in this case. 

\begin{defi}\label{defn:I_k}
For each $k \geq 0$, define the class $\sI_k$ of compact iterated edge spaces of depth $k$ as follows:
\begin{itemize} 
\item An element of $\calI_0$  is a compact smooth manifold;
\item A space $X$ lies in $\sI_k$ if it has a decomposition $X = X' \cup X''$, where $X''$ is an element of
$\sI_{k-1}$ with a codimension one boundary along the intersection $X' \cap X''$ and each component 
of $X'$ is the total space of a cone bundle over a compact base space $B$ with fibre a truncated
cone $C(Z)$ for some $Z \in \sI_{k-1}$. (The common boundary $\del X' \cap \del X''$ is the total space 
of a bundle over the same base $B$ with fibre $Z$.) 
\item Any element $X \in \sI_k$ has a well-defined dimension, where in the decomposition above, $\dim X'' = \dim X' 
= \dim B + \dim Z + 1$.
\end{itemize} 
\end{defi}
Note that if $X \in \sI_k$, then its stratum $Y$ of maximal depth $k$ is necessarily a compact smooth manifold. 

Every iterated edge space $X$ carries a class of adapted iterated edge metrics, which are also defined inductively.
Thus assuming that we have described the class of admissible iterated edge metrics on all iterated edge spaces 
of depth $k-1$, let $X$ be an iterated edge space of depth $k$. If $Y$ is the stratum of depth $k$ and
$\calU$ the tubular neighbourhood around $Y$, then we can assume that the structure of the metric $g$ 
on $X$ has been described on $X \setminus \calU$.   In particular, if $x$ is the radial coordinate on
the conic fibres of $\calU$, then $\del \calU = \{x = 1\}$ is an iterated edge space of depth $k-1$,
which is the total space of a fibration over $Y$ with fibre $Z$. Let $G$ be an admissible metric on 
$\del \cal U$, which we assume has been defined by the inductive hypothesis. Then it is of the
form $\pi^* h + k$, where $\pi: \del \calU \to Y$ is the fibration, $h$ is an ordinary Riemannian
metric on $Y$ and $k$ is a symmetric $2$-tensor on $\del \mathcal{U}$ which restricts to an admissible
metric on each fibre $Z$. We also assume that $k$ is totally degenerate on a subspace of dimension
$\ell$ at each point, $\ell = \dim Y$.  Now define the metric $g$ on $\calU$ by coning off each fibre. In other 
words, we set
\[
g_0 = dx^2 + \pi^* h + x^2 k,
\]
where $x$ is the radial function on each conical fibre. An admissible metric on $\calU$ is any metric
which has the form $g_0 + \kappa$ where $\kappa$ is polyhomogeneous on the resolution of
this neighbourhood and such that $|\kappa|_{g_0}$ decays at some rate $x^\gamma$. The notion
of polyhomogeneous regularity will be defined in the next section, so for the moment consider
this only as an appropriate smoothness condition.  It is possible to consider finite regularity
metrics too, but for simplicity we shall not do so.  Finally, a metric $g$ on all of $X$ is
admissible if it is admissible in a sense defined via the inductive hypothesis away from $\calU$,
and which takes this form in $\calU$.

To make this more explicit, let $\calV \times Z$ be a local trivialization of $\del \calU$, where $\calV 
\subset \mathbb R^\ell$ is an open ball, and $Z$ is the depth $k-1$ link, and introduce a coordinate 
system $y \in \calV$ as well as local coordinates $z$ on the smooth stratum of $Z$. Then we can write
\[
G = \sum_{i, j = 1}^\ell h_{i j} (y) dy^i dy^j + \sum_{i=1}^\ell\sum_{p=1}^{n-\ell-1} b_{ip}(y,z) dy^i dz^p + 
\sum_{p, q = 1}^{n-\ell-1} k_{pq}(y,z) dz^p dz^q,
\]
where the second and third sums here constitute the tensor $k$. Thus
\begin{multline*}
g_1 = dx^2 + \sum_{i, j = 1}^\ell h_{i j} (x,y) dy^i dy^j +  \\
x^2 \sum_{i=1}^\ell\sum_{p=1}^{n-\ell-1} b_{ip}(x, y,z) dy^i dz^p + 
x^2 \sum_{p, q = 1}^{n-\ell-1} k_{pq}(x, y,z) dz^p dz^q,
\end{multline*}
where $h_{ij}$, $b_{ip}$ and $k_{pq}$ are smooth for $0 \leq x \leq 1$. 

Pick any point $p \in Y$, the depth $k$ stratum, and use coordinates so that $x=0$, $y= 0$ and $z=z_0$
at $p$. Define the dilations $D_\lambda: (x,y,z) \to (\lambda x, \lambda y, z)$, and consider the family
of metrics $g_\lambda := D_{1/\lambda}^* g_1$. Then, as $\lambda \to \infty$, 
\[
\lambda^2 g_\lambda \longrightarrow dx^2 + \sum_{i,j=1}^\ell h_{ij}(0) dy^i dy^j + x^2 \sum_{p,q=1}^{n-\ell-1} k_{pq}(0,z) dz^p dz^q
\]
which is a product metric on the space $\RR^\ell \times C(Z)$ (where $\RR^\ell$ is identified with $T_p Y$). 
Note that the metric $k$ on $Z$ in this product
decomposition depends on the basepoint $p \in Y$, whereas $h(p)$ is simply the Euclidean metric in some
linear change of coordinates. We summarize this by saying that $|dy|^2 + k(y)$ is the model iterated
edge metric for $g$ at $y \in Y$.  Observe that the perturbation $\kappa$ disappears in this same
rescaling limit. An important consequence of this metric structure is that we can choose local coordinates $(x,y,z)$
near any fibre $Z_y$ corresponding to $y \in Y$ such that the scalar Laplacian takes the form
\[
\Delta_g = \del_x^2 + x^{-1} A(x,y,z) \del_x + x^{-2}\Delta_{k(x,y,z)} + \Delta_{h(x,y,z)} + E,
\]
where all coefficients are smooth, or at least bounded and polyhomogeneous, with $A(0,y,z) \equiv n-\ell-1$, and where 
$E$ is a higher order error term of first order in the sense that it is a sum of smooth multiples of the vector fields
$x\del_x$, $x \del_y$ and $\del_z$. 

There are slightly less restrictive types of metrics which one can handle without too much more difficulty; for example, one
could allow terms like $x dy^i dz^p$, or (for the final metric, after the perturbation $\kappa$ is added), terms like $dx dy^i$ 
or $x dx dz^p$, but again for simplicity we do not do so here.

\subsection{Sobolev inequalities}
We next show that the Sobolev inequality \eqref{Sobolev} holds on any iterated edge space with adapted metric.
\begin{prop}
Let $(M,g)$ be an iterated edge space, possibly with boundary, with admissible metric $g$ as defined in 
the last subsection. Denote by $\Omega$ its principal open dense stratum. Then the Sobolev inequality \eqref{Sobolev} 
is valid for all $u \in \calC^\infty_0(\Omega)$, and hence for all $u \in W^{1,2}_0(M)$. 
\label{siprop}
\end{prop}
\begin{proof}  We reduce the problem of verifying \eqref{Sobolev} on an iterated edge space $(M,g)$ of depth $k$ 
using the following observations.  First,  the Sobolev inequality is localizable; in other words, if \eqref{Sobolev}
holds on every set in a finite open cover $\{\calU_\alpha\}$ of $M$, then using a partition of unity we
can show that it holds on all of $M$.  Now decompose $M = M' \cup \calU$ where $\calU$ is the tubular 
neighbourhood around the maximal depth stratum in $M$ and $M'$ is an iterated edge space of depth $k-1$ 
with boundary. We may assume by induction that \eqref{Sobolev} holds for all functions with support in $M'$,
so it suffices to verify this inequality for functions with support in $\calU$. Localizing further, we can restrict
attention to functions supported in a local trivialization $\calV \times C_1(Z)$ of $\calU$, where $\calV \subset \RR^\ell$ 
is an open ball and $Z$ is a compact space of depth strictly less than $k$. Finally, noting that \eqref{Sobolev} is stable under 
quasi-isometric changes of metric, we may assume that $g$ is the product metric $|dy|^2 + dx^2 + x^2 h_Z$ on 
$\RR^\ell \times C(Z)$. 

We now recall the fact that \eqref{Sobolev} holds on a space $(W,g_W)$ if and only if the heat kernel $H^W(t,w,w')$ for the scalar 
Laplacian satisfies 
\begin{equation}
H^W(t,w,w') \leq C' t^{-n/2}
\label{heatbound}
\end{equation}
for all $w, w' \in W$ and  $0 < t < 1$, where $n = \dim W$. (Indeed, \cite[Theorem 4.1.3]{SC1} states that the Nash inequality 
is equivalent to this heat kernel estimate; the equivalence of the Nash inequality with \eqref{Sobolev} is treated in
\cite [Ch. 3]{SC1}; alternately, \cite{Nash} shows that the Sobolev inequality implies the heat kernel bound, while by
\cite{Varopoulos}, the heat kernel bound implies the Sobolev inequality.) 

We apply this in two separate ways.  First, since $Z$ is a compact iterated edge space of depth less than $k$, 
\eqref{Sobolev} holds on $Z$; hence $H^Z(t,z,z') \leq C t^{-m/2}$ where $m = \dim Z = n-\ell-1$.
Using this, we shall show that
\begin{equation}
H^{C(Z)}(t,x,z,x',z') \leq C t^{-(m+1)/2}.
\label{heatkernelcone}
\end{equation}
Since the corresponding heat kernel bound on $\RR^\ell$ is standard, and since heat kernels multiply for
Riemannian products, we see that 
\begin{multline*}
H^{\RR^\ell \times C_1(Z)}(t,y,x,z,y',x',z') =  \\
H^{\RR^\ell}(t,y, y') H^{C(Z)}(t,x,z, x',z') \leq C t^{-\ell/2 - (n-\ell)/2} = C t^{-n/2}.
\end{multline*}
Hence \eqref{Sobolev} holds on $\calV \times C_1(Z)$. 

It remains to verify \eqref{heatkernelcone}.   Denote by $H^{a,b}$ the heat kernel on the conic nappe $C_{a,b}(Z) = \{ (x,z): 
a \leq x \leq b\}$, with Dirichlet conditions at the boundaries. Note that \eqref{Sobolev} holds on $C_{1,2}(Z)$ 
with respect to the product metric, hence by quasi-isometry invariance, it also holds with respect to the conic metric.
Therefore,
\begin{equation}
H^{1,2}(t,x,z,x',z') \leq C t^{-(m+1)/2}.
\label{Sobconenap}
\end{equation}
Now recall the basic scaling property of the heat kernel.  For any $\lambda > 0$, the heat kernels $H^{\lambda, 2\lambda}$
and $H^{1,2}$ are related to one another by
\[
H^{\lambda,2\lambda}(\lambda^2 t, \lambda x, y, \lambda x', y') \lambda^{m+1} = H^{1,2}(t,x,y,x',y').
\]
Using \eqref{Sobconenap} and changing variables, we obtain
\[
H^{\lambda, 2\lambda}(t,x,y,x',y') \leq C t^{-(m+1)/2}.
\]
The squared $L^{2n/(n-2)}$ norm on the left in \eqref{Sobolev} and the squared $L^2$ norm of $u$ on 
the right both scale the same way, but the squared $L^2$ norm of $\nabla u$ scales differently. 
Thus when we apply this for the sequence $\lambda = 2^{-j}$, and assemble the pieces using a 
dyadic partition of unity $\{\chi( 2^{j}x)\}$, where $\chi$ is supported on $1/2 \leq x \leq 4$, then we 
conclude that
\[
\left( \int_{C_{0,1}(Z)} x^{-2}u^{\frac{2n}{n-2}} \, dV_g \right)^{\frac{n-2}{n}} \leq C \left( \int_{C_{0,1}(Z)} |\nabla u|^2\, dV_g +
\int_{C_{0,1}(Z)} x^{-2} u^2 \, dV_g \right),
\]
which is valid for all $u \in \calC^\infty_0(C_{0,1}(Z) \setminus \{0\})$. 

Since $x \leq 1$ in the support, the left side dominates $||u||_{2n/(n-2)}^{(n-2)/n}$. On the other hand,
we claim that there is a Poincar\'e-Hardy inequality in this setting, i.e.\ 
\[
\frac{(m-1)^2}{4} \int_{C(Z)} x^{-2} u^2 \, dV_g \leq \int_{C(Z)} |\nabla u|^2\, dV_g.
\]
This is standard when $Z$ is a compact smooth manifold, but since $\Delta_Z$ has discrete spectrum by
virtue of Corollary~\ref{cordiscspec} and the inductive hypothesis, we can reduce to the individual eigenspaces, 
where it becomes the usual Hardy inequality on $\RR^+$. As an alternate path to proving this we could use
the argument in \cite{Car-Hardy}, which uses integration by parts and hence requires only the density
of functions with compact support in the smooth locus. 

We have now verified \eqref{heatkernelcone}, and hence have proved that \eqref{Sobolev} holds for
all iterated edge spaces of depth $k$. 
\end{proof}

Appealing to Lemma~\ref{gencritsclq} below and combining the result above with Proposition~\ref{presquesobolev}, we obtain the
\begin{cor}
Let $(M,g)$ be a compact iterated edge space with adapted metric which satisfies one of the conditions in Lemma~\ref{gencritsclq}
so that at least one of the hypotheses iv) a) or iv) b) hold. Then the local Yamabe constant $Y_\ell(M,[g])$ is strictly positive. 
\label{plyi}
\end{cor}
In the next subsection we identify this local Yamabe constant somewhat more explicitly.

\subsection{Existence of Yamabe metrics}
We now turn to the problem of finding minimizers for the functional $Q_g$ in this setting of iterated edge spaces. 
The main issue now  is to understand when the hypotheses iv) a), b) or c) hold so that we can apply
Theorem~\ref{existence}. 

We first describe the local Yamabe invariant of an iterated edge space $(M,g)$. Let $p \in M$. If $p$ lies in the depth 
$0$ stratum, i.e.\ is a smooth point, then the local Yamabe invariant at $p$ is just $Y(S^n)$. If $p$ lies on a depth 
$k$ stratum $\Sigma$, then as described at the end of \S 2.1, the rescaled limit of the metric $g$ 
equals $dx^2 + dy^2 + x^2 k_p$, where $k_p$ is the metric on the link $Z$ at $p$ and $dy^2$ is the Euclidean
metric on $\RR^\ell$, $\ell = \dim \Sigma$. Note that this is conformally equivalent to the product metric 
$g_{\HH^{\ell+1}} + k_p$ on $\HH^{\ell+1} \times Z$, and hence 
\[
Y( \RR^\ell \times C(Z), dx^2 + dy^2 + x^2 k_p) = Y( \HH^{\ell+1} \times Z, [g_{\HH^{\ell+1}} + k_p]). 
\]
This generalizes  the fact that the Yamabe invariant of the cone $C(Z)$ and the cylinder $\RR \times Z$ are the same. 
In any case, enumerating the depth $j$ strata as $\{\Sigma_j\}$, and denoting the link around $\Sigma_j$ by $Z_j$, then 
we have proved that
\begin{equation}
Y_{\ell}(M,g) = \min_j \inf_{p \in Y_j} \{ Y( \RR^\ell \times C(Z_j), [dy^2 + dx^2 + x^2(k_{Z_j})_p]) \}.
\label{stratlocinv}
\end{equation}

Now consider the hypotheses in \S 1.1.  The verification of i) is a straightforward exercise using cutoff 
functions and mollifications, which we leave to the reader. The Ahlfors $n$-regularity is even easier. We have verified 
in \S 2.2 that the Sobolev inequality \eqref{Sobolev} holds; this is condition iii).  On the other hand,
the hypotheses iv) a)-c) require more careful attention. Indeed, as we now show, these hypotheses 
are valid for a rather limited set of iterated edge metrics. 

We begin with some general remarks. It is clear from the structure of adapted iterated edge metrics that if $\Sigma_j$
is any stratum of depth $j$ and $x_j$ is the radial distance function in the tubular neighbourhood of $\Sigma_j$,
then the scalar curvature $\scal_g$ can blow up no faster than $x_j^{-2}$.  If $\dim Z_j = f_j$, so $\dim \Sigma_j = \ell_j = 
n - f_j - 1$,  then $dV_g \approx x_j^{f_j} dx_j dV_{h_j} dV_{k_j}$ near $\Sigma_j$, where $h_j$ 
is a smooth metric on $\Sigma_j$ pulled back to the tubular neighbourhood and $k_j$ restricts to a metric on
the (depth $j-1$) link $Z_j$.  Assuming that $g$ is smooth in the variable $x_j$, then 
\[
\scal_g = \frac{A^{(j)}_0}{x_j^2} + \frac{A^{(j)}_1}{x_j} + \calO(1).
\]

To correlate this with the hypotheses iv) a) - c), note that $1/x_j^2 \in L^q$ implies $q < (f_j+1)/2$, and hence
we can never take $q > n/2$ as in iv) a). Similarly, the Morrey condition requires that for some $q > 1$, 
\[
r^{-n} \int_{B_r} x_j^{-2q + f_j} \, dx_j dy dz = C r^{-n + f_j - 2q + 1 + \ell_j} = C r^{-2q},
\]
which is \eqref{Morrey} with $\alpha = 2$ and hence does not fit into our hypotheses.  Suppose, however,
that the coefficient of $x_j^{-2}$ in this expansion vanishes. Then $x_j^{-1} \in L^q$ provided $q < f_j + 1$, 
and hence we can take $q > n/2$ and see that iv) a) is satisfied provided $f_j + 1 > n/2$. Similarly, the Morrey 
condition holds because we only need choose $q > 1$, which is always possible since $f_j + 1 > 1$, and for
such a $q$ we then have
\[
r^{-n} \int_{B_r} x_j^{-q + f_j} \, dx_j  dy dz = C r^{-n + f_j - q + 1 + \ell_j} = C r^{-q},
\]
which is \eqref{Morrey} with $\alpha = 1$.  This proves the
\begin{lem}
The scalar curvature $\scal_g$ satisfies iv) a) if and only if then $A^{(j)}_0 = 0$ for all $j$ and in addition $A^{(1)}_j = 0$ 
whenever $f_j \leq (n-2)/2$. The scalar curvature $\scal_g$ satisfies iv) b) (for some $q > 1$ and $0 \leq \alpha < 2$) 
if and only if $A_0^{(j)} = 0$ for all $j$. Finally, $\scal_g$ satisfies iv) c) if and only if $A_0^{(j)} \geq 0$ for all $j$ and 
$A_1^{(j)} \geq 0$ when $f_j \leq (n-2)/2$.  
\label{gencritsclq}
\end{lem}

It is clear from this that the terms in $\scal_g$ which blow up like $1/x_j^2$ are the most problematic.
The $1/x_j$ terms always fit within hypothesis iv) b). 

We will also need to consider metrics with a polyhomogeneous expansion, which include noninteger powers of $x$ or terms 
like $x^\gamma (\log x)^\ell$, $\ell \in \mathbb N_0$. For any such metric, the scalar
curvature function also has an expansion and there is an obvious extension of this lemma which requires the
vanishing or nonnegativity of the coefficient of any term $x^\gamma (\log x)^\ell$ where $n\gamma /2 + f_j \leq 1$. 

\subsubsection*{Isolated conic points}  We first examine the simplest case: an isolated conic singularity, 
where $\dim \Sigma_j = 0$ and $f_j = n$.  For simplicity drop the index $j$, but to be consistent with later 
notation, we still use $f = n-1$. A well-known formula \cite[p.69]{P} shows that an exact warped product conic metric 
$g = dx^2 + x^2 k$ has $\scal_g = x^{-2}(\scal_k - f(f-1))$. More generally, if $k$ depends smoothly on $x$, then 
\begin{equation}
\scal_g = \frac{\scal_{k(0)}  - f(f-1)}{x^2} + \calO(x^{-1})
\label{scalcurvconic}
\end{equation}
This leading coefficient vanishes if and only if $\scal_{k(0)} = f(f-1)$, which indicates a very strong geometric
and topological obstruction: if the scalar curvature of $(M,g)$ is bounded, then in particular the link 
$(Z,k(0))$ must have positive Yamabe invariant.

We now study whether it is possible to remove the singular terms in the expansion of $\scal_g$ 
using a conformal change.  If $\hat{g} = w^{\frac{4}{n-2}}g$, then
\begin{equation}
\scal_{\hat{g}} = - c(n)^{-1} w^{-\frac{n+2}{n-2}} ( \Delta_g w - c(n) \scal_g w), \quad c(n) = \frac{n-2}{4(n-1)}. 
\label{changesc}
\end{equation}
Thus if we introduce the expansion in $x$ of $\Delta_g$ and $\scal_g$, we obtain that
\begin{multline*}
\scal_{\hat{g}} \sim -c(n)^{-1}w^{- \frac{n+2}{n-2}} \times \\
\left( \del_x^2 + \frac{n-1}{x}\del_x + \frac{1}{x^2}(\Delta_{k_0}
 - c(n)(\scal_{k(0)} - f(f-1)))  + \frac{1}{x}E\right) w.
\end{multline*}
The error term $E \sim E_0 + x E_1 + \ldots$ is a second order differential operator composed of a sum of smooth
multiples of products of the vector fields $x\del_x$ and $\del_z$, and also includes the terms beyond the leading 
one in the expansion for $\scal_g$.   

From this we see that a necessary and sufficient condition for the coefficient of $x^{-2}$ to vanish is that 
\[
(\Delta_{k(0)} - c(n) \scal_{k(0)} ) w_0 = - c(n) f(f-1) w_0, 
\]
where $w_0$ is the restriction of  $w$ to $x=0$, i.e.\ the leading term in the expansion of $w$. 
We denote the operator which appears on the left here by $\calL_{k(0)}^n$; it is a special element of the family of operators 
\begin{equation}
\calL_{k_0}^m = \Delta_{k_0} - c(m) \scal_{k(0)},\quad c(m) = \frac{m-2}{4(m-1)},
\label{opfam}
\end{equation}
for any value of $m$. Note that $\calL_{k(0)}^f$ is simply the conformal Laplacian of $(Z, k(0))$. 
The positivity of the operator $-\calL_{k(0)}^n$ plays an important role in the main existence theorem of \cite{AB},
as we now recall. 
\begin{thm}[\cite{AB}]
Suppose that $(M^n,g)$ is a space with isolated conic singularities, and that at each conic point $p$, the operator
$- \calL_{k(0)}^n$ on the link $(Z,h(0))$ has all eigenvalues strictly positive. Suppose too that 
$Y(M, [g]) < Y_\ell(M,[g])$. Then there exists a function  $u$ on $M$ which 
minimizes $Q_g$ and is such that $\hat{g} = u^{\frac{4}{n-2}}g$ remains incomplete. Conversely, there exists a minimizer $u$ 
such that $\hat{g}$ is incomplete only if $-\calL_{k(0)} > 0$.
\label{AB-exist}
\end{thm}
We do not assert that $u$ is bounded, nor that the new constant scalar curvature metric is conic. We shall explain shortly
why $u$ may fail to be bounded; in the next section we describe the polyhomogeneous regularity of $u$ which makes 
clear that $\hat{g}$ is in fact still conic. 

We can recover part of this theorem immediately from Theorem~\ref{existence}. Indeed, if the lowest eigenvalue
of $-\calL_{k(0)}^n$ is {\it exactly} $c(n) f(f-1)$, then we can choose $w$ so that $w_0$ is the eigenfunction corresponding
to this lowest eigenvalue, so that $w_0$ is strictly positive, and then $\scal_{\hat{g}}$ blows up only like $x^{-1}$, hence
lies in $L^q$ for $q \in (n/2, n)$. We thus obtain the existence of a bounded, strictly positive function $u$
which minimizes $Q_{\hat{g}}$. Clearly $u^{4/(n-2)}\hat{g}$ is quasi-isometric to $\hat{g}$ and thence to $g$. 

In order to prove existence whenever $-\calL_{k(0)}^n$ is positive, fix $\delta > 0$, to be specified below, and 
define $g_\delta = x^{2\delta-2} g$. The change of variables $\xi = x^\delta/\delta$ gives the transform
\[
g_\delta = x^{2\delta - 2} (dx^2 + x^2 k) = d\xi^2 + \xi^2 \delta^2  k,
\]
so $g_\delta$ is still conic, but its link metric has been scaled by $\delta^2$.  Oserve also that 
$-\calL_{\delta^2 k(0)}^n = - \delta^{-2}\calL_{k(0)}^n$. This means that if we first replace $g$ by $g_\delta$ 
and then set $\hat{g}_\delta = w^{4/(n-2)}g_\delta$, then we can make the coefficient of $\xi^{-2}$ vanish
provided that $\delta^{-2} \lambda_0( -\calL_{k(0)}^n) = c(n) f(f-1)$, which determines the value of $\delta$,
and $w_0$ is the corresponding eigenfunction.  We are then in a position to apply Theorem~\ref{existence}
again, this time with $\hat{g}_\delta$ as the background metric. The solution metric is quasi-isometric to
$g_\delta$, and hence conic. 

The converse statement is an easy consequence of these same calculations, at least once we 
show that a minimizer $u$ (or indeed any positive solution of the corresponding Euler-Lagrange
equation) has a polyhomogeneous expansion as $x \to 0$, which we do in the next section. 
Thus we have now given an independent proof of Theorem~\ref{AB-exist}. 

The condition that $\lambda_0( -\calL_{k(0)}^n) > 0$ is actually stronger than the condition that
$(Z, k(0))$ is Yamabe positive. Indeed, referring back to the family of operators \eqref{opfam},
an easy calculation shows that if $p < q$, then there are positive constants $A = A(p,q)$ and $B =
B(p,q)$ such that
\begin{equation}
-\calL_{k(0)}^p =  A (-\calL_{k(0)}^q)+ B (-\Delta_{k(0)}) \Longrightarrow -\calL_{k(0)}^p  \geq A (-\calL_{k(0)}^q).
\label{compop}
\end{equation}
In particular, taking $p=f$ and $q=n$, then the positivity of $-\calL_{k(0)}^n$ implies that the
conformal Laplacian of $(Z, k(0))$ is positive, which is well-known to imply the existence
of a conformally equivalent (constant) positive scalar curvature metric. 

\subsubsection*{Simple edges} 
We next suppose that $M$ has only simple edges, i.e.\ that each singular stratum $\Sigma_j$ is a compact 
smooth manifold of dimension $n-r_j$. For simplicity we assume that there is only one such stratum and 
drop the index $j$. A tubular neighbourhood of $\Sigma$ is a cone bundle with compact smooth link $Z^f$, 
and in this neighbourhood, $g \sim dx^2 + x^2 k + \pi^* h$, where $h$ is a metric on $\Sigma$, $\pi$ the 
projection from this neighbourhood onto $\Sigma$ and $k$ a symmetric $2$-tensor so that $dx^2 + x^2 k$ 
pulls back to an asymptotically conic metric on each conical fibre of the tubular neighbourhood. 

\begin{lem}
If $g$ has a smooth expansion as $x \to 0$, then 
\begin{equation}
\scal_g = \frac{\scal_{k(0,y)} - f(f-1)}{x^{2}} + \frac{A_1(y,z)}{x} + \calO(1).
\label{scurvedge}
\end{equation}
\label{ssc}
\end{lem}
This is slightly less obvious than in the isolated conic case and can be verified by direct calculation. It can also be 
proved by observing that since $\scal_g$ has an expansion with initial term $x^{-2}$, if we dilate the coordinates 
via $(x,y,z) \to (\lambda x, \lambda y, z)$ (around some fixed basepoint $y_0 \in \Sigma$) and let $g_\lambda$ be the 
corresponding pulled back metric, then the coefficient of the leading term of homogeneity $-2$ in $\lambda$ must be 
the limit of $\lambda^2 \scal_{g_\lambda}$ as $\lambda \to \infty$. However, it is evident that $\lambda^{-2} g_\lambda$ 
converges to the product metric $dx^2 + x^2 k(y_0,z) + dy^2$ on $C(Z) \times \RR^{n-r}$, which has scalar curvature 
exactly equal to $x^{-2}(\scal_{k(0)} - f(f-1))$.  

Let us now investigate whether it is possible to conformally transform away the singular term of order $1/x^2$ 
in the expansion for $\scal_g$ at $\Sigma$. As we have already shown, the existence of the singular term
$A_1/x$ can be handled using hypothesis iv) b). 

Replace $g$ by $\hat{g} = w^{\frac{4}{n-2}} g$ and proceed with exactly the same formal calculation as in
the isolated conic case. There are several important differences in this setting.  First, it is still clearly 
necessary that $\lambda_0(-\calL_{k(0)(y)}^n) \equiv c(n) f (f-1)$ and that $w_0(y,z)$ must lie in this eigenspace
for every $y$. In particular, this eigenvalue must be independent of $y \in \Sigma$, which is a strong 
rigidity statement.  Assuming this, we can thus eliminate the $x^{-2}$ term.  Note that we may -- and indeed 
we shall later need to -- let $w_0$ depend nontrivially but smoothly on $y$. This does not interfere with
this calculation since although $E$ now contains $y$ derivatives, these are accompanied by a nonnegative power of $x$, 
hence the derivatives of $w_0$ can be regarded as junk terms in the expansion and can be solved away.

Applying Theorem~\ref{existence} and the fact that we have arranged that $\scal_g$ satisfies iv) b), 
we have now proved the
\begin{thm}
Let $(M,g)$ have at most simple edge singularities. Assume that $c(n)f(f-1) = \lambda_0(-\calL_{k_j(0)}^n)$  
along each singular stratum $\Sigma_j$. Suppose in addition that $Y(M,[g]) < Y_\ell(M,[g])$. Then there exists a 
bounded and strictly positive function $u$ 
which minimizes $Q_{\hat{g}}$. The metric $u^{4/(n-2)}\hat{g}$ is quasi-isometric to the initial metric $g$. 
\label{exist-se}
\end{thm}

Unlike the conic case, we cannot go further and still remain within the class of iterated edge metrics. Indeed, 
if we were to multiply $g$ by the conformal factor $x^{2\delta-2}$, then this factor would also multiply
$\pi^* h$; if $\delta < 1$, the corresponding metric would have infinite diameter, while if $\delta > 1$ 
then the entire edge $\Sigma$ would be collapsed to a point. In either case, we would
leave the category of smoothly stratified spaces and iterated edge metrics.  

We shall not carry out the detailed study of when we can modify $g$ conformally to ensure the weaker
condition iv) c), that $(\scal_{\hat{g}})_- \in L^q$ for some $q > n/2$. The conditions are not particularly 
explicit, and the solution $u$ is not bounded away from $0$ so that the solution metric is again not 
of iterated edge type. 

\subsubsection*{The general case}
We now come to the general case where $(M,g)$ is a smoothly stratified space with iterated edge metric. 
As we shall explain, the conditions needed to obtain a solution of the Yamabe problem in this category
are even more restrictive than in the simple edge case. 

We begin with a statement of the simplest case, which follows immediately from Theorem~\ref{existence}.
\begin{thm}
Let $(M,g)$ be a compact smoothly stratified space with iterated edge metric $g$. Suppose that
along each stratum $\Sigma_j$, the link metric $(Z_j, k_j)$ has $\scal_{k_j} \equiv f_j(f_j-1)$, and
in addition, that $Y(M,[g]) < Y_\ell(M,[g])$. 
Then there exists a bounded, strictly positive function $u$ which minimizes
$Q_g$, and hence $u^{\frac{4}{n-2}}g$ is an iterated edge metric with constant scalar curvature.
\label{simplestsss}
\end{thm}
The regularity theorem in the next section will show that $u$ is polyhomogeneous, so that
this solution metric is indeed an iterated edge metric in the strict sense of the word.

As in the conic and simple edge cases, we might also seek conditions on the initial metric
$g$ so that there is {\it some} conformally related metric $\hat{g}$ which satisfies the 
hypotheses of Theorem~\ref{existence}. As in those cases, the idea is to choose the conformal
factor $w$ to kill the appropriate singular terms at each stratum.

The calculations we have done above may be carried out almost exactly as before, and lead
to the following necessary conditions: for any stratum $\Sigma$ with link $(Z^f,k)$, we assume that
\begin{itemize}
\item[1)] The operator $-\calL_k^n$ on $Z$ has discrete spectrum; 
\item[2)]  The operator $-\calL_{k}^n$ has lowest eigenvalue $c(n)f(f-1)$ at every point of $\Sigma$;
\end{itemize}
The first hypothesis, on the discreteness of the spectrum, may be surprising.  The fact that the scalar Laplacian $\Delta_k$ 
itself is essentially self-adjoint and has discrete spectrum is a consequence of Corollary~\ref{cordiscspec} 
and Proposition~\ref{siprop}. However, the extra term $c(n)\scal_k$ may blow up like $1/r^2$ 
on approach to any of the singular strata of $Z$ itself, which changes the indicial roots. It is not hard to
find examples of spaces $(Z,k)$, even with just isolated conic singularities, where $-\calL_k^n$ is not
even semi-bounded, which simply amounts to the fact that $c(n)\scal_k$ diverges to $-\infty$ like $-c/r^2$ 
with leading coefficient larger $c$ than the permissible Hardy estimate bound $(f-1)^2/4$.  This question is 
closely related to the problems studied in \cite{MaMc}, see also \cite{Car-Hardy} and \cite{AK} 

One further point which requires explanation is that in using condition 2), we use a conformal factor $w$
which has leading coefficient along $\Sigma$ equal to the eigenfunction $w_0$ for $-\calL_{k}^n$ corresponding
to the eigenvalue $c(n)f(f-1)$. In order to stay with the class of iterated edge metrics, it is necessary that
$w_0$ be bounded above and strictly positive, and this may fail. Indeed, it is easy to construct
examples of operators $-\Delta_k + V$ on $Z$ with $V$ blowing up like $1/r^2$, where the ground
state eigenfunction either vanishes at the singular set of $Z$ or else blows up at some rate.  Fortunately,
the fact that this does not occur follows from the hypotheses we have already made.
\begin{prop}
If $(M,g)$ satisfies conditions 1) - 2) along each singular stratum $\Sigma$, and if $(Z,k)$ is any link,
then the eigenfunction $w_0$ for the ground state eigenvalue of $-\calL_k^n$ is bounded and strictly
positive.
\label{bddgse}
\end{prop}
\begin{proof}
The assertion follows from the regularity theory for eigenfunctions, reviewed in the next section,
and an indicial root computation. Let $\Sigma'$ be any singular stratum of $Z$ with corresponding
link $(Z', k')$, $\dim Z' = f'$. If $r$ is the radial variable to this stratum, then near $\Sigma'$, 
\[
\calL_k^n = \del_r^2 + \frac{f'}{r}\del_r + \frac{1}{r^2}(\Delta_{k'} - c(n)(\scal_{k'} - f'(f'-1))) + \Delta_{\Sigma'} + E',
\]
where $E'$ contains all higher order terms (including higher order terms in the expansion of $\scal_k$). 
The indicial roots of this operator are then equal to
\[
\nu_j^{\pm} = -\frac{f'-1}{2} \pm \sqrt{ \frac{(f'-1)^2}{4} + \mu_j},
\]
where the $\mu_j$ are the eigenvalues of $-( \Delta_{k'} - c(n)(\scal_{k'}- f'(f'-1)) = - \calL_{k'}^n - c(n)f'(f'-1)$.
By assumption, $\mu_0 = 0 < \mu_1 \leq \cdots$, hence 
\[
\nu_0^+ = 0 < \nu_1^+ < \ldots,\ \mbox{and}\ \nu_0^- = 1-f' > \nu_1^- > \ldots.
\]
By the aforementioned regularity theory, $w_0 \sim  c r^{\nu_j+} \phi_j + \ldots$ near $\Sigma'$, where
$\phi_j$ is the eigenfunction corresponding to $\nu_j^+$. However, $w_0$ must remain strictly positive
in the interior of $Z$ by the standard maximum principle arguments, hence $j = 0$ and $w_0 \sim c \phi_0$
which shows that it remains bounded and strictly positive near this stratum. 
\end{proof}

To conclude this section we observe finally that assuming the conditions 1) and 2) on $(M,g)$, 
if $(Z,k)$ is any link, then by \eqref{compop}, the conformal Laplacian $-\calL_k^f$ is strictly positive.  

\section{Regularity}
The final goal of this paper is to study the regularity of the minimizers of the functional $Q_g$ obtained in the last section when 
$(M,g)$ is an iterated edge space.  The techniques here are nonvariational, so the results below apply to any positive solution of 
\begin{equation}
\Delta_{g} u - c(n) \scal_{g} u + c(n)\Lambda u^{\frac{n+2}{n-2}} = 0,
\label{Yamabe}
\end{equation}
assuming that $u$ satisfies a natural growth condition so that $u^{\frac{4}{n-2}} g$ remains quasi-isometric to $g$,
and which is satisfied for the solutions constructed in \S 2.3. Note that if $||u||_{\frac{2n}{n-2}} = 1$, then 
$\Lambda = \scal_g \geq Y(M,[g_0])$. We shall prove that $u$ is conormal along each of the singular strata, and 
has (at least) a partially polyhomogeneous expansion. We explain this below.

As in \S 2.3, we first prove regularity when $M$ has only isolated conic singularities. The steps in this case are 
quite elementary, but rely on a certain number of definitions concerning the function spaces and the $b$-calculus 
of pseudodifferential operators. With these preliminaries, the proof of regularity in this case is only a few lines.
We then prove regularity for spaces with simple edges, and here we can quote known results about the
pseudodifferential edge calculus from \cite{Maz-edge}. For the general case we need only mimic one small
part of this edge calculus to be able to deduce what we need. 

\subsection{Conic singularities}
Our first goal is to prove the
\begin{prop}
Let $(M,g)$ be a compact space with only isolated conic singularities. Assume that $g$ is a polyhomogeneous conic metric.
Suppose that $u$ is a solution of \eqref{Yamabe} which is positive on the regular 
part of $M$ and which satisfies $u < C x^{-(n-2)/2 + \epsilon}$ for some $\epsilon > 0$ near each conic point, where 
$x$ is the radial distance to the conic tip.  Then $u$ is polyhomogeneous as $x \to 0$. If the link $(Z,k)$ satisfies the 
simplest condition, that $R_k \equiv (n-1)(n-2)$, then the expansion of $u$ takes the form $u \sim c_0 + 
c_1(z)x^{\nu_1} + \ldots$, where $c_0$ is a positive constant. If we only have that the lowest eigenvalue
of $-\calL_k^n$ is positive, then $u \sim x^{(\delta-1)(n-2)/2}(c_0(z) + c_1(z) x^{\nu_1'} + \ldots)$, where $c_0(z)$ is
strictly positive and is the ground state eigenfunction of $-\calL_k^n$ and $\delta$ is the constant
described in \S 2.3.  The exponents $\nu_1, \nu_1'$, etc., which appear in this expansion are determined
by the higher eigenvalues of $-\calL_k^n$. 
\label{regcone}
\end{prop}
This regularity is local near each conic tip, but we emphasize that it is global with respect to the
links $(Z,k)$.  The main issue is to prove that the solution is conormal (see below); its polyhomogeneity
and the precise form of its expansion are then formal consequences. 

Rather than analyzing \eqref{Yamabe} directly, we rewrite it relative to the background metric $\tilg = 
x^{-2}g = (x^{(2-n)/2})^{4/(n-2)}$, yielding 
\begin{equation}
(\Delta_{\tilg} - c(n)R_{\tilg})v + c(n)\Lambda v^{\frac{n+2}{n-2}} = 0,
\label{Yamconetrans}
\end{equation}
where the original solution $u = x^{(2-n)/2}v$. Note that by the transformation properties of the conformal Laplacian,
$R_{\tilg} = R_{k(0)} + \cdots$, hence 
\begin{equation}
\left((x\del_x)^2 + \Delta_{k(0)} - c(n) R_{k(0)} + E\right) v + c(n)\Lambda v^{\frac{n+2}{n-2}} = 0;
\label{Yamconetrans2}
\end{equation}
here $E$ contains all higher order error terms from both $\Delta_\tilg$ and $\scal_\tilg$. Note that our assumption
that $|u| \leq C x^{(2-n)/2 + \epsilon}$ becomes $|v| \leq C x^\e$, which is much easier to work with. (In fact,
it is precisely because we are able to work with solutions which decay that this argument is easier
than the corresponding regularity theorem in \cite{Maz-rsyp}.) 

Before embarking on the proof, we recall several facts, first about the function spaces which will be used
and then about parametrices in the $b$-calculus. For simplicity, assume that $M$ has only one conic point 
and that the radial function $x$ is extended globally and is strictly positive elsewhere on $M$.  
\begin{defi}
Decompose $M$ as $M' \sqcup C_1(Z)$, where the second factor is the truncated cone over $Z$ with 
coordinates $z \in Z$ and $x \in (0,1]$. It is most natural to work relative to the complete metric
$\tilg$, and in this geometry, $-\log x$ is the distance function on the asymptotically cylindrical end. 
\begin{itemize}
\item[i)]
The space $\calC^{k,\gamma}_b(M)$ consists of all functions $v$ which lie in the ordinary H\"older space 
$\calC^{k,\gamma}$ on $M'$, and in addition satisfy $(x\del_x)^j \del_z^\alpha v \in \calC^{0,\gamma}_b$, 
where the latter space is defined using the seminorm
\[
[ v ]_{b; 0,\gamma} := \sup_{ (x,z) \neq (x',z') \atop 1/2 \leq x/x' \leq 2} \frac{ |v(x,z) - v(x',z')|}{\mbox{dist}_{\tilg}((x,z),(x',z'))^\gamma}.
\]
We also define
\[
x^\mu \calC^{k,\gamma}(M) = \{ v = x^\mu \tilde{v}: \tilde{v} \in \calC^{k,\gamma}_b(M)\}.
\]
\item[ii)] For any $\nu \in \RR$, let $\calA^\nu (M) = \bigcap_{k \geq 0} x^\nu \calC^{k,\gamma}(M)$.  This is the space of
conormal functions. Next, define the space of polyhomogeneous functions $\calA_{\phg}(M)$ to consist of all conormal
functions $v$ which admit complete asymptotic expansions with smooth coefficients, and 
write $\calA_{\phg}^\nu$ for all polyhomogeneous functions with leading term $x^{\nu_0} \phi(z)$ for some
$\nu_0$ with (real part) greater than or equal to $\nu$. Note that $x^\nu \log x \in \calA_{\phg}^{\nu-\epsilon}$ for
any $\epsilon > 0$.  Finally, let $\nu < \nu'$ be any pair of real numbers, and define $\calA_{\phg}^{\nu, \nu'}(M) =
\calA_{\phg}^\nu(M) + \calA^{\nu'}(M)$; thus $v$ is in this space if it has a partial polyhomogeneous expansion with initial
term bounded by $x^\nu$ and with conormal `remainder' vanishing like $x^{\nu'}$. 
\end{itemize}
\end{defi}

As a first step in the proof of Proposition~\ref{regcone}, note that since $\tilg$ has locally uniformly controlled 
geometry and since $v$ is uniformly bounded, we obtain directly from classical H\"older estimates that
$v \in \calA^\e(M)$. 

Using this in \eqref{Yamconetrans2}, we have
\begin{equation}
((x\del_x)^2 + \Delta_{k(0)} -c(n) R_{k(0)}) v = -E v + c(n) \Lambda v^{\frac{n+2}{n-2}} \in \calA^{\tau \e},
\label{indr}
\end{equation}
where $\tau = \frac{n+2}{n-2}$ (and we assume that $\tau \e \leq 1$ for simplicity). 

The conformal Laplacian $L_\tilg$ is an example of an elliptic $b$-operator, and the operator
on the left in this last equation is its asymptotic model at $x=0$ and called its indicial
operator, $I(L_{\tilg})$. This indicial operator can be analyzed quite directly using the Mellin 
transform in the $x$ variable. To this end, introduce the indicial roots
of $L_\tilg$; these are the values $\nu$ for which there exists a function $\phi \in \calC^\infty(Z)$ such that
\[
I(L_{\tilg})  x^\nu \phi = 0 \Leftrightarrow  L_{\tilg} x^\nu \phi = \calO(x^{\nu + \epsilon'})
\]
for some $\e' > 0$. It is easy to see in this case, by separation of variables, that the indicial roots are given by
\begin{equation}
\nu_j^\pm := \pm \sqrt{\lambda_j}, \ \ \mbox{where}\ \mbox{spec}(-\calL_{k(0)}) = \{\lambda_j\}.
\label{indrts}
\end{equation} 
The coefficient function $\phi$ for any such indicial root equals the corresponding eigenfunction $\phi_j$. 
Since the lowest eigenvalue of $-\calL_{k(0)}^n = - \Delta_{k(0)} + c(n) R_{k(0)}$ is strictly positive, we 
have that $\ldots \leq \nu_1^- < \nu_0^- < 0 < \nu_0^+ < \nu_1^+ \leq \ldots $.  The indicial roots are 
the precise rates of growth or decay of approximate solutions of $L_\tilg w = 0$.   

The $b$-calculus is merely a systematized method for passing from information about the indicial operator
to the corresponding information about $L_\tilg$ itself. We quote some results from this theory, referring to 
\cite{Maz-edge} for a careful development of this $b$-calculus as well as the more general edge calculus which will
be invoked below. 
\begin{prop} {\rm (\cite[Theorem 4.4]{Maz-edge})} 
For $k \in \mathbb N$ and $0 < \gamma < 1$, the mapping
\[
L_\tilg:  x^\nu \calC^{k+2,\gamma}_b(M) \longrightarrow x^{\nu} \calC^{k,\gamma}_b(M)
\]
is Fredholm if and only if $\nu \neq \nu_j^\pm$ for any $j$. 
\label{bprop1}
\end{prop}
\begin{prop}{\rm (\cite[Proposition 3.28]{Maz-edge})} 
Let $f \in \calA^{\nu'}(M)$ and suppose that $L_\tilg v = f$, where $v \in \calA^\nu(M)$ for some $\nu < \nu'$. Then
$v \in \calA_{\phg}^{\nu,\nu'}(M)$, or in other words, $v$ has a partial expansion
\[
v = \sum_{j=0}^N  \sum_{p=0}^{N_j'} x^{\mu_j} (\log x)^p v_{jp}(z) + \tilde{v},
\]
where $\tilde{v} \in \calA^{\nu'}$ and the $\mu_j$ lie in the interval $(\nu, \nu')$.  Moreover, if $f \in \calA_{\phg}^{\nu, \nu'}(M)$,
then $v \in \calA_{\phg}^{\nu, \nu'}(M)$ and if $f \in \calA_\phg^{\nu}$, then $v \in \calA_{\phg}^\nu$.
\label{bprop2}
\end{prop}
\begin{rem}  Once we know that $v \in \calA^{\nu, \nu'}_\phg$ for some $\nu, \nu'$, we can determine the exponents $\mu_j$
which appear in its expansion by a formal computation with the equation $L_{\tilg} v = f$. In particular, if the link metric
$k(x)$ depends smoothly on $x$, then all $\mu_j$ are of the form $\nu_j^\pm + \ell$, where $\nu_j^\pm$ is an 
indicial root and $\ell \in \NN_0$. 
\end{rem}

These  are proved by constructing a parametrix $G$ for $L_\tilg$, which is a pseudodifferential operator, 
depending on the choice of (nonindicial!) weight $\nu$. The fundamental mapping results for this class of
operators, proved in \cite [\S 3]{Maz-edge} give that
\begin{eqnarray}
G: x^{\nu} \calC^{k-2,\gamma}_b(M) & \longrightarrow x^\nu \calC^{k,\gamma}_b(M), \label{mp1}\\
G: \calA^{\nu'}(M) & \longrightarrow \calA_{\phg}^{\nu, \nu'}(M), \label{mp2} \\
G: \calA^{\nu}_{\phg}(M) & \longrightarrow \calA_{\phg}^\nu(M) \label{mp3}
\end{eqnarray}
are all bounded mappings, for any $k \in \NN_0$, and $\nu < \nu'$ with $\nu \notin \{\nu_j^\pm\}$. 

\medskip

\noindent{\it Proof of Proposition~\ref{regcone}:}  Appealing directly to \eqref{indr} and applying Proposition~\ref{bprop2},
we deduce that $v \in \calA_{\phg}^{\e, \tau \e}$.  Using this on the right side of this equation gives
$v \in \calA_{\phg}^{\e, \tau^2 \e}$, and bootstrapping further, we obtain that $v \in \calA_{\phg}$. 
The precise form of its expansion can then be determined by substituting this expansion into the equation.
In particular, the leading exponent in the expansion is equal to one of the positive indicial roots $\nu_j^+$, i.e.\ 
\[
v \sim \phi_j(z) x^{\nu_j^+} + \calO(x^{\nu_j^+ + \e'}).
\]
Since $v > 0$ when $x > 0$ and the eigenfunction $\phi_j$ changes sign unless $j = 0$, this expansion
must start with $\phi_0(z) x^{\nu_0^+}$. 

In the simplest case, where $R_{k(0)} \equiv (n-1)(n-2)$, we can easily see that $\nu_0^+ = (n-2)/2$
and $\phi_0(z)$ is a positive constant.  In particular, $u = x^{(2-n)/2} v$ is bounded and strictly positive.
In the more general case where we only assume that $-\calL_{k(0)}^n > 0$, we conclude that
$u \sim \phi_0(z) x^{ (2-n)/2 + \nu_0^+}$, where now $\phi_0$ is variable but still strictly positive.
Recalling the conformal change $g \mapsto g_\delta = x^{2\delta-2}g$ introduced in \S 2.3
which allows one to reduce to the case that $\nu_0^+ = (n-2)/2$ with respect to the
new radial coordinate $\xi = x^\delta/\delta$ provided one chooses $\delta$ correctly, we
have that the solution is bounded and strictly positive relative to the background metric $g_\delta$. 
Note that in the first case, the solution metric is still exact conic, while in the second case,
it is conformally exact conic. 
\hfill \qed

\medskip

\subsection{Simple edges}
We next present the corresponding proof of regularity when $(M,g)$ has simple edges. In this case we can take 
advantage of the construction of parametrices and their mapping properties in the edge calculus.  There is 
one important new step in the proof, beyond what was needed in the conic case, but then it proceeds
exactly as before. 

We begin with exactly the same initial conformal change, replacing $g$ by $\tilg = x^{-2}g$, with the
corresponding changes of conformal Laplacian and the Yamabe equation.  We write the solution
metric as $v^{4/(n-2)} \tilg$, so that the initial hypothesis is that $0 \leq v \leq C x^\e$. 
Using the uniform geometry of $\tilg$ we deduce immediately that $v$ is infinitely differentiable
with respect to the geometry of this metric, or equivalently, in local coordinates,
\[
|(x\del_x)^j (x\del_y)^\alpha \del_z^\beta v | \leq C_{j, \alpha, \beta} x^\e
\]
for all $j, \alpha, \beta$. Note however that this is not sufficient to assert conormality yet because these
estimates do not control the $y$ derivatives, i.e.\ the derivatives tangent to the singular stratum, as $x \to 0$.
Obtaining this control requires the parametrix for the conformal Laplacian $L_{\tilg}$ as constructed
in \cite{Maz-edge}.  The difference between this operator and the one for the conic problem is
the inclusion of the tangential part of the Laplacian $x^2 \Delta_y$; thus when rewriting the Yamabe
equation in the form \eqref{indr}, the error term $E$ on the left includes $x^2 \Delta_y v$, which at
this stage we cannot guarantee vanishes any faster than $x^\e$. 

The indicial roots of $L_\tilg$ are defined exactly as in the conic case above, and in particular
are given by exactly the same formul\ae, with the important difference that the eigenvalues
$\lambda_j$ of $-\calL_{k(0)}^n$ can vary with $y$.  Because of this, we do not expect solutions to
have discrete asymptotics, i.e.\ polyhomogeneity, in the sense above, and we shall be satisfied
with a more limited partial polyhomogeneity result. Recall, however, our key hypothesis
that the lowest eigenvalue of $-\calL_{k(0)}^n$ is equal to $c(n)f(f-1)$, and in particular is
independent of $y$. Let $\nu_0^\pm = \pm \sqrt{c(n)f(f-1)}$ be the corresponding indicial roots.
Fix $0 < \nu < \nu_0^+$. By \cite[Theorem 6.1]{Maz-edge}, there exists a parametrix $G$
for $L_{\tilg}$ with the properties:   $G \circ L = \mbox{Id} - Q$, where $Q$ is a finite 
rank operator which maps into $\calA^{\nu_0^+ + \e'}$, and the analogues of \eqref{mp1}, 
\eqref{mp2} and \eqref{mp3} are all valid provided $\nu' \leq \nu_1^+$ and furthermore, that 
we replace $\calC^{k,\gamma}_b$ by the H\"older spaces $\calC^{k,\gamma}_e$ based on derivatives 
with respect to $x\del_x$, $x\del_y$ and $\del_z$.  The final extra property we need is
that the commutator $[\del_y, G] = G_1$ is an operator in the edge calculus with exactly
the same mapping properties as $G$ itself; this is \cite[Theorem 3.30]{Maz-edge}. 

We now use all of this information as follows. First let us apply the parametrix $G$ 
to $L_\tilg v + c(n)\Lambda v^{(n+2)/(n-2)} = 0$, to obtain that
\[
v = G ( c(n)\Lambda v^{\frac{n+2}{n-2}}) + Qv.
\]
We wish to use this equation to prove that $v$ is conormal. Once this is done, the existence of
the partial expansion follows directly from \eqref{mp2} and bootstrapping.  Applying $\del_y$ to 
each side gives
\[
\del_y v = G ( c(n)\Lambda \del_y v^{\frac{n+2}{n-2}}) + G_1( c(n) \Lambda v^{\frac{n+2}{n-2}}) + \del_y Qv. 
\]
The final term $Qv$ is already conormal so presents no difficulties here. Applying \eqref{mp1} 
to $G_1$, we see that this term lies in $x^\e \calC^{k,\gamma}_e$ for all $k$. Finally, write $\del_y v^{(n+2)/(n-2)} = 
c v^{4/(n-2)}\del_y v$, and recall that $v^{4/(n-2)} \leq C x^{4\e/(n-2)}$. In the simplest version of
this argument, $\e \geq (n-2)/4$ so that $v \leq Cx$. Note that this is satisfied if we assume
that our original solution $u$ is bounded, which is a more natural assumption once we 
leave the setting of isolated singularities; for simplicity we assume that this is the case. 
Then $|v^{4/(n-2)}\del_y v| \leq |x\del_y v|$, hence this term lies in $\cap_k x^\e \calC^{k,\gamma}_e$,
and finally, $\del_y v \in \cap_k x^\e \calC^{k,\gamma}_e$. Iterating this argument gives eventually that 
$\del_y^\alpha \in \cap_k x^\e \calC^{k,\gamma}_e$ for every multi-index $\alpha$, so that
$v \in \calA^\e$.   Proceeding as explained above shows that $v$, and hence $u$, has a partial 
expansion. 

\subsection{The general case}  
The final step in our proof of regularity is to extend these arguments to handle the case when $M$
is a general smoothly stratified space and $g$ is an iterated edge metric satisfying the hypotheses
i) - iii) of \S 2.3.  We also assume that the solution $u$ is bounded. We shall be rather brief here 
since we have covered almost all of the main points of the argument already.

Suppose that $M$ is a space of depth $k$. By induction and the fact that this regularity
theorem is localizable, we may assume that $u$ has the appropriate regularity, i.e.\ partial
polyhomogeneity, everywhere except possibly along the singular strata of highest depth,
and so we can focus on these. Indeed, we can focus on the equation \eqref{Yamabe} 
in a neighbourhood of the form $\calU \times C_1(Z)$ where $\calU \subset \RR^\ell_y$,
$Z$ is a compact smoothly stratified space of depth $k-1$, where again the regularity
result is known by induction, and $C_1(Z)$ is the truncated cone over $Z$.  Let $x$ be the radial
variable on this cone. As before, we conformally transform this problem, writing the
equation in terms of the new partially completed background metric $\tilg = x^{-2}g$,
so $u = x^{(2-n)/2}v$.  As an important part of our inductive hypothesis, we assume the
operator $-\calL_{k(0)}^n$ on $Z$ has discrete spectrum.  This allows us to define
indicial roots and analyze the indicial operator exactly as before. 

We use function spaces $\calC^{k,\gamma}_{\ice}$ based on derivatives with respect to $x\del_x$, 
$x\del_y$ and vector fields $V$ which are tangent to the fibre $Z$ and all of its singular strata.  
We have, by local elliptic regularity and induction, that  $v \in x^{(n-2)/2}\calC^{k,\gamma}_{\ice}$ 
for all $k$, so the remaining job is to prove that $v$ is conormal, that every tangential
derivative $\del_y^\alpha$ decays at this same rate.  We do this by applying exactly the same
commutator argument, but unfortunately there is no ready-made iterated edge calculus to
which we can appeal. Fortunately we need very few consequences of such a calculus and 
can deduce these from a somewhat primitive parametrix construction.  This is carried out
in more detail in \cite{ALMP}.  Recall that we wish to construct an operator $G$ such
that, with $L$ equal to the conformal Laplacian for $\tilg$,  $G L = I - Q$ where $Q$
maps into $\calA^\nu(M)$, which has mapping properties analogous to \eqref{mp1}, \eqref{mp2},
\eqref{mp3}, and finally, so that the commutator $[\del_y, G]$ enjoys the same mapping properties. 

For simplicity let $N$ denote the localized space $\calU \times C_1(Z)$ blown up
fibrewise at the vertex of each cone. Thus $N$ is the product $\calU \times [0,1) \times Z$. 
We construct $G$ just as in \cite{Maz-edge} by regarding its Schwartz kernel
$G(x,y,z, \tilde{x}, \tilde{y}, \tilde{z})$ as a distribution on the space $N^2_{\ice}$ obtained from
$N \times N$ by blowing up the fibre diagonal at the corner $\{x = \tilde{x} = 0\}$. 
In fact, this space is identified with the space $(\calU \times [0,1))^2_0 \times Z \times Z$, where
the first factor is the $0$-double space of $\calU \times [0,1)$, as constructed and
used in \cite{Maz-edge}. This has three boundary components, the left and right
faces, corresponding to $\tilde{x} \to 0$ and $\tilde{x} \to 0$, and the front
face $\mathrm{ff}$, which covers $\{x = \tilde{x} = 0, y = \tilde{y}\}$ and is the face 
created in the blowup.  The key point is that the lift of $G$ to this space is conormal
and partially polyhomogeneous at all faces. Its leading coefficient at $\mathrm{ff}$ is
precisely the inverse for the so-called normal operator 
\[
N(L) = (x\del_x)^2 + x^2 \Delta_y + \calL_{k(0)}^n
\]
which is globally defined on $\RR^+ \times \RR^\ell \times Z$. The invertibility of
this normal operator on $x^\nu \calC^{k,\gamma}_{\ice}$ is the main ellipticity hypothesis,
and is proved exactly as in the simple edge case, relying on the fact that $-\calL_{k(0)}^n$
is strictly positive.  We analyze this normal operator by taking the Fourier transform in $y$,
thus reducing it to 
\[
(x\del_x)^2 + \calL_{k(0)}^n - x^2 |\eta|^2,
\]
and then rescaling, setting $t = x|\eta|$, to arrive at
\[
(t\del_t)^2 + \calL_{k(0)}^n - t^2,
\]
which is an operator on $\RR^+ \times Z$. The inverse for this can be analyzed as in \cite{Maz-edge},
using mainly that $\calL_{k(0)}^n$ has discrete spectrum. We denote the Schwartz kernel of this inverse
by $\hat{G}(t, z, \tilde{t}, \tilde{z})$. Rescaling and taking the inverse Fourier transform, we see that
the Schwartz kernel of the inverse of the normal operator equals
\[
\int_{\RR^\ell}  e^{i(y-\tilde{y})\eta} G_0(x|\eta|, \tilde{x} |\eta|, z, \tilde{z}) |\eta|^q\, d\eta
\]
for an appropriately chosen $q$. 

The mapping properties for this inverse are deduced exactly as in the simple edge case. The final
fact concerning the commutator $[\del_y, G]$ can now be proved just as in the edge calculus, 
where this ultimately reduces to the fact that the commutator of the globally defined
translation-invariant vector field $\del_y$ on $\RR^\ell$ commutes exactly with
the explicit inverse for the normal operator written above. 

We have been (extremely) sketchy in the development in this last section. There are several
reasons for this. The first is simply that while the idea is very close to that used in
the simple edge case, it would still take considerable space to write out these details
fully, and given the relatively minor importance of this final result, we have chosen not to do so.
The sketch above is intended to provide a guide for anyone with a reasonable familiarity
with the edge calculus.  Finally, we point out that there are certainly other proofs that
one might carry out to prove this regularity which would be more elementary in the
sense that they do not explicitly use blowups and pseudodifferential operators, but
which would require a substantial amount of verification of elementary details nonetheless.

As described in the introduction, in a companion piece to this paper we give the full details
of a proof of rather different sort of regularity statement which requires very little regularity
of the background iterated edge metric, and which shows that the solution $u$ to the 
Yamabe equation enjoys some H\"older continuity properties.

\end{document}